\documentclass[11pt]{article}
\usepackage{amssymb,amsmath,graphicx,amsthm,mathrsfs}
\usepackage[usenames]{color}
\newtheorem{thm}{Theorem}[section]

\theoremstyle{definition}

\theoremstyle{remark}

\numberwithin{equation}{section}

\title{\bf Enhanced Approximate Cloaking by SH and FSH Lining}

\author{Jingzhi Li\thanks{SAM, D-MATH, ETH Z\"urich, CH--8902, Switzerland ({\tt
jingzhi.li@sam.math.ethz.ch})}\qquad Hongyu Liu\thanks{Department of
Mathematics and Statistics, University of North Carolina, Charlotte,
NC 28263, USA. ({\tt hliu28@uncc.edu})} \qquad Hongpeng
Sun\thanks{Institute of Mathematics, Academy of Mathematics and
Systems Science, Chinese Academy of Sciences, Beijing 100190, P. R.
China. ({\tt hpsun@amss.ac.cn})} }

\begin{document}

\date{}

\maketitle

\begin{abstract}
We consider approximate cloaking from a regularization viewpoint introduced in \cite{KSVW} for EIT and further investigated in
\cite{KOVW,Liu} for the Helmholtz equation. The cloaking schemes in \cite{KOVW} and \cite{Liu} are shown to be (optimally) within
$|\ln\rho|^{-1}$ in 2D and $\rho$ in 3D of perfect cloaking, where $\rho$ denotes the regularization
parameter. In this paper, we show that by employing a sound-hard layer
right outside the cloaked region, one could (optimally) achieve
$\rho^N$ in $\mathbb{R}^N,\ N\geq 2$, which significantly enhances the near-cloak. We then develop a cloaking scheme
by making use of a lossy layer with well-chosen parameters. The lossy-layer cloaking scheme is shown to possess the same cloaking
performance as the one with a sound-hard layer. Moreover, it is shown that the lossy layer could be taken as a finite realization of the sound-hard layer. Numerical experiments are also presented to assess the cloaking performances of all the cloaking schemes for comparisons.
\end{abstract}


\section{Introduction}\label{sec:intro}

A region is said to be \emph{cloaked} if its contents together with
the cloak are invisible to certain measurements. From a practical
viewpoint, these measurements are made in the exterior of the cloak.
Blueprints for making objects invisible to electromagnetic waves
were proposed by Pendry {\it et al.} \cite{PenSchSmi} and Leonhardt
\cite{Leo} in 2006. In the case of electrostatics, the same idea was
discussed by Greenleaf {\it et al.} \cite{GLU2} in 2003. The key
ingredient is that optical parameters have transformation properties
and could be {\it pushed-forward} to form new material parameters.
The obtained materials/media are called {\it transformation media}.
We refer to \cite{CC,GKLU4,GKLU5,Nor,U2,YYQ} for state-of-the-art
surveys on the rapidly growing literature and many striking
applications of the so-called `transformation optics'.

In this work, we shall be mainly concerned with the acoustic
cloaking via `transformation acoustics'. The transformation media
proposed in \cite{GLU2,PenSchSmi} are rather singular. This poses
much challenge to both theoretical analysis and practical
fabrication. In order to avoid the singular structures, several
regularized approximate cloaking schemes are proposed in
\cite{GKLU_2,KOVW,KSVW,Liu,RYNQ}. The basic idea is to introduce
regularization into the singular transformation underlying the ideal
cloaking, and instead of the perfect invisibility, one would
consider the `near-invisibility' depending on a regularization
parameter. Our study is closely related to the one introduced in
\cite{KSVW} for approximate cloaking in EIT, where the
`blow-up-a-point' transformation in \cite{GLU2,PenSchSmi} is
regularized to be the `blow-up-a-small-region' transformation. The
idea was further explored in \cite{KOVW} and \cite{Liu} for the
Helmholtz equation. In \cite{Liu}, the author imposed a homogeneous
Dirichlet boundary condition at the inner edge of the cloak and
showed that the `blow-up-a-small-region' construction gives
successful near-cloak. In \cite{KOVW}, the authors introduced a
special lossy-layer between the cloaked region and the cloaking
region, and also showed that the `blow-up-a-small-region'
construction gives successful near-cloak. For both cloaking
constructions, it was shown that the near-cloaks come, respectively,
within $1/|\ln \rho|$ in 2D and $\rho$ in 3D of the perfect
cloaking, where $\rho$ is the relative size of the small region
being blown-up for the construction and plays the role of a
regularization parameter. These estimates are also shown to be
optimal for their constructions.

It is worth noting that if one lets the loss parameter in
\cite{KOVW} go to infinity, this limit corresponds to the imposition
of a homogeneous Dirichlet boundary condition at the inner edge of
the cloak. On the other hand, the imposition of a homogeneous
Dirichlet boundary condition at the inner edge of the cloak is
equivalent to employing a sound-soft layer right outside the cloaked
region. In this sense, the lossy layer lining in \cite{KOVW} is a
finite realization of the sound-soft lining in \cite{Liu}. We would
like to emphasize that employing some special lining is necessary
for the near-cloak construction, since otherwise it is shown in
\cite{KOVW} that there exists certain resonant inclusion which
defies any attempt to achieve near-cloak.

In this work, we shall impose a homogeneous Neumann boundary
condition on the inner edge of the cloak, which amounts to employing
a sound-hard layer right outside the cloaked region. The cloaking
scheme is referred to as an {\it SH construction}. For the SH
construction, we show that one could achieve significantly enhanced
cloaking performance. Actually, it is shown that one could achieve,
respectively, $\rho^2$ in 2D and $\rho^3$ in 3D within the perfect
cloaking for such construction. We then develop a cloaking scheme by
making use of a well-chosen lossy layer lining. The properly
designed lossy-layer could be taken as a finite realization of the
sound-hard lining. The cloaking scheme is referred to as an {\it FSH
construction}. The FSH construction is shown to possess the same
cloaking performance as the SH construction.

The analysis of cloaking must specify the type of exterior
measurements. In \cite{GKLU_2,KOVW,KSVW}, the near-cloaks are
assessed in terms of boundary measurement encoded into the boundary
Dirichlet-to-Neumann (DtN) map. The scattering measurement is
considered for the near-cloak in \cite{Liu}. In the present work, we
shall assess our near-cloak construction with respect to scattering
measurement encoded into the scattering amplitude. Nonetheless, by
\cite{Nach,NSU}, it can be shown, at least heuristically, that
knowing the scattering amplitude amounts to knowing the boundary DtN
map.

In this paper, we focus entirely on transformation-optics-approach
in constructing cloaking devices. But we would
like to mention in passing the other promising cloaking schemes
including the one based on anomalous localized resonance \cite{MN},
and another one based on special (object-dependent) coatings
\cite{AE}. It is also interesting to note a recent work in \cite{AKLL}, where the authors
implement multi-coatings to enhance the near-cloak in EIT.

The rest of the paper is organized as follows. In Section 2, we
develop the cloaking scheme by employing the sound-hard lining. In
Section 3, we present the cloaking scheme with a properly designed
lossy layer. Section 4 is devoted to discussions on different cloaking
schemes.  In Section 5, we present the numerical examples.

\section{Transformation acoustics and cloaking construction}

Let $q\in L^\infty(\mathbb{R}^N)$ be a scalar function and $\sigma=(\sigma^{ij})_{i,j=1}^N\in\mbox{Sym}(N)$ be a
symmetric-matrix-valued function on $\mathbb{R}^N$, which is bounded in
the sense that, for some constants $0<c_0<C_0<\infty$,
\begin{equation}
\label{eqn:Bound_Sigma} c_0 \xi^T \xi \leq \xi^T \sigma(x) \xi \leq
C_0 \xi^T \xi
\end{equation}
for all $x\in \mathbb{R}^N$ and $\xi \in \mathbb{R}^N$. In acoustics, $\sigma^{-1}$ and $q$, respectively, represent the mass density tensor and the bulk modulus of a {\it
regular} acoustic medium. We shall denote
$\{\mathbb{R}^N; \sigma, q\}$ an acoustic medium as described above.
It is assumed that the inhomogeneity of the acoustic medium is
compactly supported, namely, $\sigma=I$ and $q=1$ in
$\mathbb{R}^N\backslash\bar{\Omega}$ with $\Omega$ a bounded Lipschitz domain
in $\mathbb{R}^N$. In $\mathbb{R}^N$, the time-harmonic acoustic
wave propagation is governed by the heterogeneous Helmholtz equation
\begin{equation}\label{eq:Helmholtz II}
\mbox{div}(\sigma\nabla u)+k^2 q u=0,
\end{equation}
where $k>0$ represents the wave number. Stationary scattering theory is to seek a solution to (\ref{eq:Helmholtz II}) admitting the following
asymptotic development
\begin{equation}\label{eq:scattering amplitude}
u(x)=e^{ix\cdot\xi}+\frac{e^{ik|x|}}{|x|^{(N-1)/2}}\left\{A\left(\frac{x}{|x|}, \frac{\xi}{|\xi|}\right)+\mathcal{O}\left(\frac{1}{|x|}\right)\right\},\ \ |x|\rightarrow\infty,
\end{equation}
where $\xi=k d$ with $d\in\mathbb{S}^{N-1}$. $A(\hat{x},d)$ with $\hat{x}:=x/|x|$ is the so-called scattering amplitude. An important problem arising in practical application is to recover $\{\Omega; \sigma, q\}$ from the measurement of the corresponding scattering amplitude. In the following, for clarity and also the convenience of our subsequent study, we give a more detailed description of the scattering problem. We shall let $u^{i}(x):=e^{ikx\cdot d}$ denote a time-harmonic plane wave, where $d\in\mathbb{S}^{N-1}$ denotes the incident direction. Let $u^{int}$ and $u^{ext}$ denote the total wave fields inside and outside the inhomogeneous medium, respectively, which satisfy the following PDE system
\begin{equation}\label{eq:Hel system}
\begin{cases}
& \nabla\cdot(\sigma\nabla u^{int})+k^2 q u^{int}=0\quad\mbox{in\ \ $\Omega$},\\
& \Delta u^{ext}+k^2 u^{ext}=0\quad\mbox{in\ \ $\mathbb{R}^N\backslash\bar{\Omega}$},\\
& \displaystyle{u^{int}|_{\partial \Omega}=u^{ext}|_{\partial\Omega},\quad \sum_{i,j=1}^n \nu_i\sigma^{ij}\partial_j u^{int}\bigg|_{\partial\Omega}=\frac{\partial u^{ext}}{\partial\nu}\bigg|_{\partial\Omega},}\\
& u^{ext}(x)=u^i(x)+u^s(x),\ \ x\in\mathbb{R}^N\backslash\bar{\Omega},\\
& \displaystyle{\lim_{r\rightarrow\infty}r^{(N-1)/2}\left\{\frac{\partial u^s}{\partial r}-ik u^s\right\}=0,}
\end{cases}
\end{equation}
where $r=|x|$ for $x\in\mathbb{R}^N$. We know $u^s\in H_{loc}^1(\mathbb{R}^N)$ (see, e.g. \cite{Isa,McL}) and clearly, $A(\hat{x},d)$ can be read off from the large $|x|$ asymptotics of $u^s$.

In this paper, we shall be concerned with the construction of a layer of cloaking medium which makes the inside scatterer invisible to scattering amplitude. To that end, we present a quick discussion on transformation acoustics. Let $\tilde
x=F(x):\Omega\rightarrow\widetilde\Omega$ be a bi-Lipschitz and
orientation-preserving mapping. For an acoustic medium
$\{\Omega;\sigma,q\}$, we let the {\it push-forwarded} medium be
defined as
\begin{equation}\label{eq:pushforward}
\{\widetilde\Omega;\widetilde\sigma,\widetilde
q\}=F_*\{\Omega;\sigma,q\}:=\{\Omega; F_*\sigma, F_*q\},
\end{equation}
where
\begin{equation}
\begin{split}
&\widetilde{\sigma}(\tilde
x)=F_*\sigma(x):=\frac{1}{J}M\sigma(x)M^T|_{x=F^{-1}(\tilde x)}\\
&\widetilde{q}(\tilde x)=F_*q(x):=q(x)/J|_{x=F^{-1}(\tilde x)}
\end{split}
\end{equation}
and $M=(\partial \tilde{x}_i/\partial x_j)_{i,j=1}^n$,
$J=\mbox{det}(M)$. Then $u\in H^1(\Omega)$ solves the Helmholtz equation
\[
\nabla\cdot(\sigma\nabla u)+k^2q u=0\quad\mbox{on\ $\Omega$},
\]
if and only if the pull-back field $\widetilde u=(F^{-1})^*u:=u\circ F^{-1}\in H^1(\widetilde{\Omega})$ solves
\[
\widetilde{\nabla}\cdot(\widetilde\sigma\widetilde\nabla \widetilde
u)+k^2\widetilde q\widetilde u=0.
\]
We have made use of $\nabla$ and $\widetilde\nabla$ to distinguish the
differentiations respectively in $x$- and $\tilde x$-coordinates. We
refer to \cite{KOVW, Liu} for a proof of this invariance.

We are in a position to construct the cloaking device. In the sequel, let $D\Subset\Omega$ be a Lipschitz domain such that $\Omega\backslash\bar{D}$
is connected. W.L.O.G., we assume that $D$ contains the origin. Let $\rho>0$ be sufficiently small and $D_\rho:=\{\rho x; x\in D\}$. Suppose
\begin{equation}\label{eq:trans}
F_\rho: \bar{\Omega}\backslash D_\rho\rightarrow \bar{\Omega}\backslash D,
\end{equation}
which is a bi-Lipschitz and orientation-preserving mapping, and $F_\rho|_{\partial \Omega}=\mbox{Identity}$. A celebrated example of such blow-up mapping is given by
\begin{equation}\label{eq:F:ball:map}
y=F_\rho(x):=\left(\frac{R_1-\rho}{R_2-\rho}R_2+\frac{R_2-R_1}{R_2-\rho}|x|\right)\frac{x}{|x|},\
\ \rho<R_1<R_2
\end{equation}
which blows-up the central ball $B_\rho$ to $B_{R_1}$ within $B_{R_2}$. Now, we set
\begin{equation}\label{eq:cloaking medium}
\{\Omega\backslash\bar{D};\sigma_c^\rho,q_c^\rho\}=(F_\rho)_*\{\Omega\backslash\bar{D}_{\rho};
I, 1\}.
\end{equation}
Let $M\Subset D$ represent the cloaked region. Then we claim the
following construction gives a near-cloaking device
\begin{equation}\label{eq:partial cloaking device}
\{\mathbb{R}^N;\sigma,q\}=\begin{cases}
 \ \ \{I, 1\}\hspace*{3.6cm}& \mbox{in\ \ $\mathbb{R}^N\backslash\bar{\Omega}$};\\
\ \ \{\sigma^\rho_c, q^\rho_c\} & \mbox{in\ \
$\Omega\backslash\bar{D}$};\\
\ \ \mbox{a sound-hard layer} & \mbox{in\ \
$D\backslash M$};\\
\ \ \mbox{arbitrary target object} & \mbox{in\ \ $M$}.
\end{cases}
\end{equation}
In (\ref{eq:partial cloaking device}), by a sound-hard layer we mean
a layer of material which prevents acoustic wave from penetrating
inside and the normal velocity of the underlying wave field vanishes
on the exterior boundary of the layer. The wave
equation governing the wave scattering corresponding to the cloaking
device constructed in (\ref{eq:partial cloaking device}) is
\begin{equation}\label{eq:physical wave}
\begin{cases}
&\nabla\cdot(\sigma\nabla u)+k^2 q u=0\quad \mbox{in\ \ $\mathbb{R}^N\backslash\bar{D}$},\\
&\displaystyle{\sum_{i,j=1}^N(\sigma_c^\rho)^{ij}\nu_i\partial_j
u=0\quad \mbox{on\ \ $\partial D$}},
\end{cases}
\end{equation}
where $\nu=(\nu_i)_{i=1}^N$ is the exterior unit normal vector to
$\partial D$. In the following, we shall let $\mathcal{A}(\hat{x},d)$ denote the scattering amplitude to the PDE system
(\ref{eq:physical wave}) corresponding to the cloaking device. Let
\[
F=F_\rho\quad\mbox{on\ \ $\Omega\backslash\bar{D}_{\rho}$};\ \ \mbox{Identity}\ \ \mbox{on\ \ $\mathbb{R}^N\backslash\Omega$}.
\]
Set $v=F^*u\in H_{loc}^1(\mathbb{R}^N\backslash\bar{D}_\rho)$ and $v^s(x):=v(x)-e^{ikx\cdot d}$. By transformation acoustics, together with
straightforward calculations, one can show that
\begin{equation}\label{eq:virtual}
\begin{cases}
& (\Delta+k^2)v=0\quad\mbox{in\ \ $\mathbb{R}^N\backslash\bar{D}_\rho$},\\
& \displaystyle{\frac{\partial v}{\partial \nu}\bigg|_{\partial D_\rho}=0},\\
& v(x)=v^s(x)+e^{ix\cdot\xi}\quad x\in\mathbb{R}^N\backslash\bar{D}_\rho,\\
& \displaystyle{\lim_{r\rightarrow\infty} r^{(N-1)/2}\left\{\frac{\partial v^s}{\partial r}-ik v^s\right\}=0.}
\end{cases}
\end{equation}
In terms of the terminologies in \cite{Liu}, (\ref{eq:physical wave}) describes the scattering in the physical space and (\ref{eq:virtual}) describes the scattering in the virtual space. Since $v=u$ in $\mathbb{R}^N\backslash\bar{\Omega}$, we see the scattering in the physical space is the same as that in the virtual space. That is, $\mathcal{A}(\hat{x},d)$ could also be read off from the large $|x|$ asymptotics of $v^s$. Next, we give one of the main results of this paper, which justifies the near-invisibility of the above construction.

\begin{thm}\label{thm:main}
There exists $\rho_0>0$ such that when $\rho<\rho_0$
the scattering amplitude $\mathcal{A}(\hat{x},d)$ to (\ref{eq:virtual}) satisfies
\begin{equation}\label{eq:full cloaking}
|\mathcal{A}(\hat{x},d)|\leq C\rho^N,\ \ \ \hat{x},d\in\mathbb{S}^{N-1},
\end{equation}
where $C$ is a constant dependent only on $k, \rho_0$ and $D$, but completely independent of $\rho$.
\end{thm}

\begin{proof}
In order to ease the exposition, we shall only prove the theorem for $N=2,3$. But we would like to emphasize that for $N>3$, the proof follows by completely similar arguments.

Let
\begin{equation}\label{eq:fun}
G(x)=\frac{i}{4}\left(\frac{k}{2\pi|x|}\right)^{(N-2)/2}H_{(N-2)/2}^{(1)}(k|x|)
\end{equation}
be the outgoing Green's function. By Green's representation, we know for $x\in
\mathbb{R}^N\backslash\bar{D}_\rho$
\begin{equation}\label{eq:green formula}
\begin{split}
v^s(x)=&\int_{\partial D_\rho}\left\{\frac{\partial G(x-y)}{\partial\nu(y)}v^s(y)-G(x-y)\frac{\partial
v^s(y)}{\partial\nu(y)}\right\}\ ds(y)\\
=& (\mathcal{K} v^s)(x)+g(x),
\end{split}
\end{equation}
where we have set
\begin{align}
(\mathcal{K} v^s)(x):=&\int_{\partial D_\rho}\frac{\partial G(x-y)}{\partial\nu(y)}v^s(y)\ d s(y),\label{eq:kernel}\\
g(x):=& -\int_{\partial D_\rho} G(x-y)\frac{\partial
v^s(y)}{\partial\nu(y)}\ ds(y).\label{eq:boundary function}
\end{align}
Clearly, $v^s(x)|_{\partial D_\rho}\in H^{1/2}(\partial D_\rho)$.
By the jump properties of the double-layer potential operator $\mathcal{K}$ (cf.
\cite{McL}), we have from (\ref{eq:green formula})
\begin{equation}\label{eq:integral equation 1}
\frac 1 2 v^s(x)=(\mathcal{K}v^s)(x)+g(x),\ \ \ x\in\partial D_\rho.
\end{equation}
Let $x'=x/\rho$, then (\ref{eq:integral equation 1}) is read as
\begin{equation}\label{eq:integral equation 2}
\frac 1 2 v^s(\rho x')=(\mathcal{K}v^s)(\rho x')+g(\rho x'),\quad x'\in\partial D.
\end{equation}

Next, we claim
\begin{equation}\label{eq:estimate g}
\|g(\rho\ \cdot)\|_{L^2(\partial D)}\leq C\rho,
\end{equation}
where $C$ remains uniform as $\rho\rightarrow 0^+$. In the sequel, we shall make use of the following asymptotic developments
of the 2D $G(x)$ (cf. \cite{ColKre}),
\begin{equation}\label{eq:asym fundamental}
G(x)=-\frac{1}{2\pi}\ln |x|+\frac{i}{4}-\frac{1}{2\pi}\ln\frac{k}{2}
-\frac{E}{2\pi}+\mathcal{O}
\left(|x|^2\ln |x|\right)
\end{equation}
for $|x|\rightarrow 0$, where $E$ is the Euler's constant. Moreover, we know that
\begin{equation}\label{eq:f1}
G(x)=\frac{e^{ik|x|}}{4\pi|x|}\ \ \mbox{when\ $N=3$}.
\end{equation}
In order to prove (\ref{eq:estimate g}), we first assume $x\in\partial D_{t\rho}$ with $1<t\leq 2$. Then by Green's formula, we have
\begin{equation}\label{eq:ee}
\begin{split}
&\int_{\partial D_{\rho}} G(x-y)\frac{\partial e^{iky\cdot d}}{\partial\nu(y)}\ ds(y)\\
=&\int_{D_\rho}\Delta_ye^{iky\cdot d} G(x-y)\ dy+\int_{D_\rho}\nabla_y G(x-y)\cdot\nabla_y e^{iky\cdot d}\ dy\\
=&-k^2\int_{D_\rho}e^{iky\cdot d} G(x-y)\ dy+\int_{D_\rho}\nabla_y G(x-y)\cdot\nabla_ye^{iky\cdot d}\ dy.
\end{split}
\end{equation}
Using (\ref{eq:asym fundamental}) and (\ref{eq:f1}), we have for $x'\in D_t$
\[
\begin{split}
|g_1(\rho x')|&=k^2|\int_{D_\rho}e^{iky\cdot d} G(\rho x'-y)\ dy|\\
&\leq k^2 \int_{D} |G(\rho (x'-y'))\rho^n dy'|.
\end{split}
\]
By the mapping properties of volume potential operator, one has
\begin{equation}\label{eq:ae1}
|g_1(x)|_{C(\partial D_{t\rho})}\leq C\rho,
\end{equation}
where $C$ is independent of $\rho$ and $t$. In like manner, one can show that
\begin{equation}\label{eq:ae2}
|g_2(x)|_{C(\partial D_{t\rho})}=|\int_{D_\rho}\nabla_y G(x-y)\cdot\nabla_ye^{iky\cdot d}\ dy|_{C(\partial D_{t\rho})}\leq C\rho.
\end{equation}
By (\ref{eq:ae1}) and (\ref{eq:ae2}), we see
\begin{equation}\label{eq:ae3}
|g(x)|_{C(D_{t\rho})}\leq C \rho.
\end{equation}
Next, by the mapping property of single-layer potential operator, we know
\begin{equation}\label{eq:ee2}
g(x)|_{\partial D_\rho}=\lim_{t\rightarrow 1^+}\left(g(x)|_{\partial D_{t\rho }}\right),
\end{equation}
which together with (\ref{eq:ae3}) implies (\ref{eq:estimate g}).

We proceed to the integral equation (\ref{eq:integral equation 1}). First, by using change of variables in integrals, it is straightforward to show that
\begin{equation}\label{eq:decompose}
(\mathcal{K} v^s)(\rho x')=(\mathcal{K}_0 v^s(\rho\ \cdot))(\rho x')+(R v^s(\rho\ \cdot))(\rho x'),\quad x'\in \partial D,
\end{equation}
where $\mathcal{K}_0$ is an integral operator with the kernel given by
\[
G_0(x-y)=\begin{cases}
 \displaystyle{-\frac{1}{2\pi}}\ln|x-y|\quad & N=2,\\
 \displaystyle{\frac{1}{4\pi}\frac{1}{|x-y|}}\quad & N=3,
\end{cases}
\]
which is the fundamental solution to $-\Delta$; and $R$ satisfies
\begin{equation}
\|R\|_{\mathcal{L}(L^2(D), L^2(D))}\lesssim \begin{cases}
& \rho\ln\rho\ \ \mbox{when $N=2$};\\
& \rho\qquad\ \mbox{when $N=3$}.
\end{cases}
\end{equation}
Hence, the integral equation (\ref{eq:integral equation 2}) can be reformulated as
\begin{equation}\label{eq:integral equation 3}
\left[(\frac{1}{2} I-\mathcal{K}_0-R) v^s(\rho\ \cdot)\right](\rho x')=g(\rho x')\quad x'\in D.
\end{equation}
By the well-known result in \cite{Ver}, $I-\frac 1 2 \mathcal{K}_0$ is invertible from $L^2(\partial D)$ to $L^2(\partial D)$. Then, by using
(\ref{eq:estimate g}), we have from (\ref{eq:integral equation 3}) that
\begin{equation}\label{eq:estimate phi}
\|v^s(\rho\ \cdot)\|_{L^2(\partial D)}\leq C \|g(\rho\ \cdot)\|_{L^2(\partial D)}\leq C\rho.
\end{equation}
Noting $\|v^s(\cdot)\|_{L^2(\partial D_\rho)}=\rho^{(N-1)/2}\|v^s(\rho\ \cdot)\|_{L^2(\partial D)}$, we further have from (\ref{eq:estimate phi})
that
\begin{equation}\label{eq:estim}
\|v^s\|_{L^2(\partial D_\rho)}\leq
\begin{cases}
C\rho^{3/2}\quad & N=2,\\
C\rho^2\ \ & N=3.
\end{cases}
\end{equation}

Finally, by letting $|x|\rightarrow \infty$ in (\ref{eq:green formula}), we have
\begin{equation}\label{eq:far field}
\mathcal{A}(\hat{x}, d)=\gamma\int_{\partial D_\rho}\left[\frac{\partial e^{-i k\hat{x}\cdot y}}{\partial\nu(y)}v^s(y)+e^{-i k\hat{x}\cdot y}\frac{\partial e^{iky\cdot d}}{\partial\nu(y)}\right]\ ds(y),
\end{equation}
where $\gamma=e^{i\frac\pi 4}/\sqrt{8\pi k}$ when $N=2$, and
$\gamma=1/4\pi$ when $N=3$. By using (\ref{eq:estim}) and the
Schwarz inequality in (\ref{eq:far field}), we have
\begin{equation}\label{eq:eee1}
\left|\int_{\partial D_\rho}\frac{\partial e^{-i k\hat{x}\cdot y}}{\partial\nu(y)} v^s(y)\ ds(y)\right|\leq C\rho^N.
\end{equation}
By Green's formula, we have
\begin{equation}\label{eq:eee2}
\begin{split}
&\left|\int_{\partial D_\rho}e^{-ik\hat{x}\cdot y}\frac{\partial e^{iky\cdot d}}{\partial\nu(y)}\ ds(y)\right|\\
=& \left|\int_{D_\rho}(\Delta e^{iky\cdot d})e^{-ik\hat{x}\cdot y}\ dy+\int_{D_\rho}\nabla e^{iky\cdot d}\cdot\nabla e^{-ik\hat{x}\cdot y}\ dy\right|\\
=& |k^2(d\cdot\hat{x}-1)\int_{D_\rho} e^{ik(d-\hat{x})\cdot y}\ dy| \\
\leq & C\rho^N.
\end{split}
\end{equation}
By (\ref{eq:far field}), (\ref{eq:eee1}) and (\ref{eq:eee2}), we have (\ref{eq:full cloaking}).

The proof is completed.

\end{proof}

By Theorem~\ref{thm:main}, we know the construction (\ref{eq:partial cloaking device}) gives a near-invisibility cloaking within $\rho^N$ of the perfect cloaking. For a special case by taking $D_\rho=B_\rho$, namely, the central ball of radius $\rho>0$, using wave functions expansion, one can show (cf. \cite{LiuIMA})
\begin{equation}\label{eq:ff1}
\mathcal{A}(\hat{x},d)=-e^{-i\frac\pi 4}\sqrt{\frac{2}{\pi k}}\left[\frac{J_0'(k\rho)}{{H_0^{(1)}}'(k\rho)}+2\sum_{n=1}^\infty\frac{J_n'(k\rho)}{{H_n^{(1)}}'(k\rho)}\cos n\theta\right]
\end{equation}
in $\mathbb{R}^2$, where $\theta=\angle(\hat{x},d)$; and
\begin{equation}\label{eq:ff2}
\mathcal{A}(\hat{x},d)=\frac{i}{k}\sum_{n=0}^\infty (2n+1)\frac{j_n'(k\rho)}{{h_n^{(1)}}'(k\rho)} P_n(\cos\theta)
\end{equation}
in $\mathbb{R}^3$, where $P_n$ is the Legendre polynomial of degree $n$. By the asymptotic developments of spherical Bessel functions (cf. \cite{LiuIMA}), one has from (\ref{eq:ff2})
\begin{equation}\label{eq:b1}
\mathcal{A}(\hat{x},d)=\frac{i}{k}\left(\frac{\cos\theta}{2}-\frac 1 3\right)\left(k\rho\right)^3+\mathcal{O}((k\rho)^5).
\end{equation}
Similarly, by (\ref{eq:ff1}) one can show that in $\mathbb{R}^2$,
\begin{equation}\label{eq:b2}
\mathcal{A}(\hat{x},d)=-e^{-i\frac\pi 4}\sqrt{\frac{2\pi}{k}}\left(\frac{\cos\theta}{2}-\frac 1 4\right)(k\rho)^2+\mathcal{O}((k\rho)^4).
\end{equation}
By (\ref{eq:b1}) and (\ref{eq:b2}), it is readily seen that the estimates in Theorem~\ref{thm:main} are optimal for full scattering measurements, namely, $\hat{x}\in\mathbb{S}^{N-1}$ and $d\in\mathbb{S}^{N-1}$. Nevertheless, it is interesting to note from (\ref{eq:b1}) and (\ref{eq:b2}) that for some specific scattering measurements, e.g., $\mathcal{A}(\hat{x},d)$ with $\angle(\hat{x},d)=\pm \arccos\frac 2 3$ in 3D, and with $\angle(\hat{x},d)=\pm\frac{\pi}{3}$ in 2D, one would have even more enhanced invisibility cloaking effects. Physically speaking, the cloaking effect would be stronger in the backward scattering region, i.e. $|\angle(\hat{x},d)|>\pi/2$, than that in the forward scattering region, i.e. $|\angle(\hat{x},d)|>\pi/2$. This could be partly seen from (\ref{eq:far field})--(\ref{eq:eee2}), and (\ref{eq:b1})--(\ref{eq:b2}).

\section{Near-cloak construction with a lossy layer}

In this section, we shall develop a lossy approximate cloaking scheme. To that end, we first introduce the following transformation $T:\mathbb{R}^N\rightarrow\mathbb{R}^N$,
\begin{equation}\label{eq:trans2}
y=T(x):=\begin{cases}
\qquad x\qquad &\mbox{for\ \ $x\in\mathbb{R}^N\backslash\bar{\Omega}$},\\
\quad F_\rho(x) & \mbox{for\ \ $x\in\Omega\backslash D_\rho$},\\
\qquad \frac{x}{\rho} & \mbox{for\ \  $x\in D_\rho$},
\end{cases}
\end{equation}
where $F_\rho$ is given in (\ref{eq:trans}). Let
\begin{equation}\label{eq:phy1}
\{\mathbb{R}^N; \sigma, q\}=\begin{cases}
\qquad I, 1\qquad & \mbox{in\ \ $\mathbb{R}^N\backslash\Omega$},\\
\quad T_* I, T_* 1 & \mbox{in\ \ $\Omega\backslash D$},\\
\quad T_*\sigma_l, T_*q_l & \mbox{in\ \ $D\backslash D_{1/2}$},\\
\quad \sigma_a', q_a' & \mbox{in\ \ $D_{1/2}$},
\end{cases}
\end{equation}
be the cloaking device in the physical space. Here, $\{D_{1/2}; \sigma_a', q_a'\}$ is an arbitrary but regular medium, which
represents the target object being cloaked; and
\begin{equation}\label{eq:lossy}
\{D_{\rho}\backslash D_{\rho/2}; \sigma_l, q_l\}
\end{equation}
is a lossy layer whose parameters shall be specified in the following. Similar to our earlier argument in Section 3, by transformation acoustics we see that the scattering amplitude in the physical space corresponding to the cloaking device is the same as the one in the virtual space. In the virtual space, the wave scattering is governed by the following PDE system
\begin{equation}\label{eq:virt1}
\begin{cases}
(\Delta+k^2) u=0\quad & \mbox{in\ \ $\mathbb{R}^N\backslash \bar{D}_\rho$},\\
\nabla\cdot(\sigma_l\nabla u)+k^2 q_l u=0\quad & \mbox{in\ \ $D_\rho\backslash \bar{D}_{\rho/2}$},\\
\nabla\cdot(\sigma_a\nabla u)+k^2 q_a u=0\quad & \mbox{in\ \ $D_{\rho/2}$},
\end{cases}
\end{equation}
where
\begin{equation}\label{eq:cloaked contents}
\{D_{\rho/2}; \sigma_a, q_a\}=(T^{-1})_*\{D_{1/2}; \sigma_a', q_a'\}
\end{equation}
is arbitrary but regular; and $u\in H_{loc}^1(\mathbb{R}^N)$ satisfies
\begin{align*}
& u(x)=e^{ikx\cdot d}+u^s(x)\quad\mbox{for\ \ $x\in\mathbb{R}^N\backslash D_{\rho}$},\\
& \displaystyle{\lim_{r\rightarrow \infty} r^{(N-1)/2}\left\{\frac{\partial u^s}{\partial r}-ik u^s\right\}=0.}
\end{align*}
We shall choose
\begin{equation}\label{eq:lossy para}
\sigma_l=C \rho^{2+2\delta}I\quad\mbox{and}\quad q_l=a+i b
\end{equation}
with $C, \delta, a, b$ some fixed positive constants in our construction. We shall show that it will yield an enhanced approximate cloaking scheme. However, the proof for the general case with general geometry and arbitrary cloaked contents is rather technical and lengthy, which we choose to extend to full details in our forthcoming paper \cite{LiLiuSun}. In this section, we shall consider the special case with spherical geometry and uniform cloaked contents. In the sequel, we let $D_\rho=B_\rho$ and, $\sigma_a'$ and $q_a'$ be arbitrary positive constants but independent of $\rho$. With a bit abusing of notation, we shall also write
\[
\sigma_l=C\rho^{2+2\delta}.
\]
By (\ref{eq:cloaked contents}), one can show by direct calculations that in the virtual space
\[
\sigma_{a}, q_a = \begin{cases} \sigma_{a}',\frac{q_{a}'}{\rho^{2}} &\mbox{in $D_{\rho/2}$ when $N=2$} \\
\frac{\sigma_{a}'}{\rho}, \frac{q_{a}'}{\rho^{3}} &\mbox{in $D_{\rho/2}$ when $N=3$}
\end{cases}.
\]

For the PDE system (\ref{eq:virt1}), we let $u = u_{0}$ in $\mathbb{R}^N\backslash \bar{D}_{\rho}$, $u = u_{2}$ in $D_{\rho}\backslash {\bar{D}}_{\rho/2}$ and $u = u_{int}$ in $D_{\rho/2}$. By transmission conditions, we have
\begin{equation}\label{boundary:big}
u_{2}=u_{0}\quad\mbox{and}\quad \sigma_{l}\frac{\partial
u_{2}}{\partial \nu}=\frac{\partial u_{0}}{\partial \nu}\quad
\mbox{on\ $\partial D_\rho$},
\end{equation}
and
\begin{equation}\label{boundary:small}
u_{2}=u_{int}\quad\mbox{and}\quad
\sigma_{l}\frac{\partial u_{2}}{\partial \nu} = \sigma_{a} \frac{\partial u_{int}}{\partial \nu}\quad\mbox{on\ $\partial D_{\rho/2}$}.
\end{equation}
Set $\tilde{k} = k \sqrt{\frac{q_{l}}{\sigma_{l}}}$ and $k_{2} =
k\sqrt{\frac{q_{a}}{\sigma_{a}}}$. We choose the complex branch of
$\tilde{k}$ such that $\Im(\tilde{k})>0$.

We first consider the 2D case. By \cite{ColKre}, one has the
following series expansions
\begin{equation}\label{eq:series}
\begin{cases}
\displaystyle{u_{0}(x) = e^{ikx \cdot d} + u^{s} =  \sum_{n = -\infty}^{\infty}i^{n}J_{n} (k|x|)e^{in\theta} + \sum_{n = -\infty}^{\infty} d_{n}H_{n}^{(1)}(k|x|)e^{in\theta}}, \\
\displaystyle{u_{2}(x) = \sum_{n = -\infty}^{\infty}a_{n}J_{n}(\tilde{k}|x|)e^{in \theta} + \sum_{n = -\infty}^{\infty}b_{n}H_{n}^{(1)}(\tilde{k}|x|)e^{in\theta}},\\
\displaystyle{u_{int} = \sum_{n = -\infty}^{\infty}c_{n}J_{n}(k_{2}|x|)e^{in \theta}.}
\end{cases}
\end{equation}

By (\ref{boundary:big}) and (\ref{eq:series}), we have
\begin{equation}\label{eq:t1}
\begin{cases}
a_{n}J_{n}(\tilde{k}\rho) + b_{n} H_{n}^{(1)}(\tilde{k}\rho)  = i^{n} J_{n}(k\rho) + d_{n}H_{n}^{(1)}(k\rho) \\
\sigma_{l} [a_{n} \tilde{k} J_{n}'(\tilde{k}\rho) + b_{n} \tilde{k} {H_{n}^{(1)}}'(\tilde{k}\rho)]=i^{n}kJ_{n}'(k\rho) + d_{n}k {H_{n}^{(1)}}'(k\rho).
\end{cases}
\end{equation}
In like manner, by (\ref{boundary:small}) and (\ref{eq:series}) we have
\begin{equation}\label{eq:t2}
\begin{cases}
a_{n}J_{n}(\tilde{k} \rho/2) + b_{n} H_{n}{(1)}(\tilde{k}\rho/2) = c_{n} J_{n}(k_{2}\rho/2) \\
\sigma_{l} [a_{n}\tilde{k}J_{n}'(\tilde{k} \rho/2) + b_{n} \tilde{k} {H_{n}^{(1)}}'(\tilde{k}\rho/2)] = \sigma_{a} c_{n} k_{2} J_{n}'(k_{2}\rho/2).
\end{cases}
\end{equation}
Here $J_{n}'(\tilde{k}\rho) = \frac{dJ_{n}(z)}{dz}|_{z = \tilde{k}\rho}$ and $J_{n}'(\tilde{k}\rho/2) = \frac{dJ_{n}(z)}{dz}|_{z = \tilde{k}\rho/2}$. ${H_{n}^{(1)}}'(\tilde{k}\rho)$, ${H_{n}^{(1)}}'(\tilde{k}\rho/2)$, $J_{n}'(k\rho)$, ${H_{n}^{(1)}}'(k\rho)$,
$J_{n}'(k_{2}\rho/2)$, ${H_{n}^{(1)}}'(k_{2}\rho/2)$ are understood in the same sense.
By letting $C_{0} = 1/{\sqrt{\sigma_{l}q_{l}}}$ and $A = \sqrt{q_{a}\sigma_{a}} = \sqrt{q_{a}'\sigma_{a}'}/\rho$, $k_{2} = \frac{k}{\rho}\sqrt{\frac{q_{a}'}{\sigma_{a}'}}$, (\ref{eq:t1}) and (\ref{eq:t2}) are read as
\begin{equation}\label{boundary1:big}
\begin{cases}
a_{n}J_{n}(\tilde{k}\rho) + b_{n} H_{n}^{(1)}(\tilde{k}\rho)  = i^{n} J_{n}(k\rho) + d_{n}H_{n}^{(1)}(k\rho) \\
[a_{n}  J_{n}'(\tilde{k}\rho) + b_{n} {H_{n}^{(1)}}'(\tilde{k}\rho)]=C_{0}[i^{n}J_{n}'(k\rho) + d_{n} {H_{n}^{(1)}}'(k\rho)],
\end{cases}
\end{equation}
and
\begin{equation}\label{boundary1:small}
\begin{cases}
a_{n}J_{n}(\tilde{k} \rho/2) + b_{n} H_{n}^{(1)}(\tilde{k}\rho/2) = c_{n} J_{n}(k_{2}\rho/2) \\
[a_{n}J_{n}'(\tilde{k} \rho/2) + b_{n}{H_{n}^{(1)}}'(\tilde{k}\rho/2)] =C_{0} A c_{n} J_{n}'(k_{2}\rho/2).
\end{cases}
\end{equation}
Noting $k_2\rho=k\sqrt{q_a'/\sigma_a'}$, if $J_n(k_2\rho/2)\neq 0$ then by direct calculations, we have from (\ref{boundary1:small})
\[
c_{n} = \frac{a_{n}J_{n}(\tilde{k}\rho/2) + b_{n} H_{n}^{(1)}(\tilde{k}\rho/2)}{J_{n}(k_{2}\rho/2)}
\]
and
\begin{equation}\label{an:bn}
b_{n} = - \frac{J_{n}'(\tilde{k}\rho/2) - C_{0} A \frac{J_{n}'(k_{2}\rho/2)}{J_{n}(k_{2}\rho/2)} J_{n}(\tilde{k}\rho/2)}
{{H_{n}^{(1)}}'(\tilde{k}\rho/2) - C_{0} A \frac{J_{n}'(k_{2}\rho/2)}{J_{n}(k_{2}\rho/2)} H_{n}^{(1)}(\tilde{k}\rho/2)} a_{n}.
\end{equation}
In case $J_n(k_2\rho/2)=0$, we have from the first equation in (\ref{boundary1:small}) that
\begin{equation}\label{eq:nn1}
b_n=-\frac{J_n(\tilde{k}\rho/2)}{H_n^{(1)}(\tilde{k}\rho/2)}a_n.
\end{equation}
Let $\Upsilon_0$ denote the fraction in (\ref{an:bn}) or (\ref{eq:nn1}) and hence $b_{n} = \Upsilon_{0} a_{n}$. Plugging $b_n$ into (\ref{boundary1:big}), we have by
straightforward calculations
$$
a_{n} = \frac{i^{n} J_{n}(k\rho) + d_{n}H_{n}^{(1)}(k\rho)}{J_{n}(\tilde{k}\rho) + \Upsilon_{0} H_{n}^{(1)}(\tilde{k}\rho)},
$$
and
\begin{equation}\label{equation:dn}
d_{n} = - \frac{i^{n}J_{n}'(k \rho) - \frac{1}{C_{0}} \frac{J_{n}'(\tilde{k}\rho) + \Upsilon_{0} {H_{n}^{(1)}}'(\tilde{k}\rho)}{J_{n}(\tilde{k}\rho) + \Upsilon_{0} H_{n}^{(1)}(\tilde{k}\rho)} i^{n}J_{n}(k\rho)}
{{H_{n}^{(1)}}'(k \rho) - \frac{1}{C_{0}} \frac{J_{n}'(\tilde{k}\rho) + \Upsilon_{0} {H_{n}^{(1)}}'(\tilde{k}\rho)}{J_{n}(\tilde{k}\rho) + \Upsilon_{0} H_{n}^{(1)}(\tilde{k}\rho)} H_{n}^{(1)}(k\rho)}.
\end{equation}

We next investigate the asymptotic development of
\begin{equation}\label{equation:h}
\mathcal{H}(\sigma_{l}, \rho) := \frac{1}{C_{0}} \frac{J_{n}'(\tilde{k}\rho) + \Upsilon_{0} {H_{n}^{(1)}}'(\tilde{k}\rho)}{J_{n}(\tilde{k}\rho) + \Upsilon_{0} H_{n}^{(1)}(\tilde{k}\rho)}=\frac{\frac{1}{C_{0}} \frac{J_{n}'(\tilde{k}\rho)}{J_{n}(\tilde{k}\rho)} + \frac{1}{C_{0}}\Upsilon_{0} \frac{{H_{n}^{(1)}}'(\tilde{k}\rho)}{J_{n}(\tilde{k}\rho)}}{1 + \Upsilon_{0} \frac{H_{n}^{(1)}(\tilde{k}\rho)}{J_{n}(\tilde{k}\rho)}}.
\end{equation}
Clearly, we only need study the asymptotic behaviors of $\frac{1}{C_{0}} \frac{J_{n}'(\tilde{k}\rho)}{J_{n}(\tilde{k}\rho)}$, $ \frac{1}{C_{0}}\Upsilon_{0} \frac{{H_{n}^{(1)}}'(\tilde{k}\rho)}{J_{n}(\tilde{k}\rho)}$ and $\Upsilon_{0} \frac{H_{n}^{(1)}(\tilde{k}\rho)}{J_{n}(\tilde{k}\rho)}$. We note the following fact due to (\ref{eq:lossy para}),
\[
\Re(\tilde{k}\rho) \rightarrow +\infty\quad\mbox{and}\quad \Im(\tilde{k}\rho) \rightarrow +\infty\quad\mbox{as\ \ $\rho\rightarrow 0^+$},
\]
which implies the following asymptotic developments for $z=\tilde{k}\rho$ (see formulas 9.2.1, 9.2.3 in \cite{AI}),
\begin{equation}\label{eq:a1}
\begin{cases}
J_{n}(z) \sim \sqrt{\frac{2}{\pi z}} \cos(z - \frac{n\pi}{2} - \frac{\pi}{4}) + e^{|\Im(z)|}\mathcal{O}(|z|^{-1}),  \ \  |\arg{z}| < \pi  \\
H_{n}^{(1)}(z) \sim \sqrt{\frac{2}{\pi z}} e^{i(z-\frac{n\pi}{2} - \frac{\pi}{4})} , \ \  -\pi < \arg{z} < 2 \pi
\end{cases}
\end{equation}
We also note the following recurrence relation (see formula 9.1.27 in \cite{AI}) for the subsequent use,
\begin{equation}\label{eq:r1}
\mathcal{B}_{n}'(z) = \frac{n}{z}\mathcal{B}_{n}(z) - \mathcal{B}_{n+1}(z)
\end{equation}
where $\mathcal{B}_{n}(z)=J_{n}(z)$ or $H_{n}^{(1)}(z)$. Since $J_{n}(z)=(-1)^{n}J_{n}(z)$ and $H_{n}^{(1)}(z)=$$(-1)^{n}H_{-n}^{(1)}(z)$ (see formula 9.1.5 in \cite{AI}), we only need consider the case with $n\geq 0$ in the sequel.

Noting $\cos(z)=\frac{e^{iz}+e^{-iz}}{2}$, by (\ref{eq:a1}) we further have as $\Re(z), \Im(z)\rightarrow\infty$
\begin{equation}\label{asym:special}
\begin{cases}
J_{n}(z)  \sim \sqrt{\frac{1}{2\pi z}} e^{|\Im(z)|}e^{i(-\Re(z)+\frac{n\pi}{2} + \frac{\pi}{4})},  \ \  |\arg{z}| < \pi  \\
H_{n}^{(1)}(z) \sim \sqrt{\frac{2}{\pi z}} e^{-\Im(z)}e^{i(\Re(z)-\frac{n\pi}{2} - \frac{\pi}{4})}, \ \  -\pi < \arg{z} < 2 \pi \\
|{H_{n}^{(1)}}'(z)| \sim \sqrt{\frac{2}{\pi |z|}} e^{-\Im(z)},\ \  -\pi < \arg{z} < 2 \pi
\end{cases}
\end{equation}
It is remarked that the third equation in (\ref{asym:special}) is obtained by using the recurrence relation (\ref{eq:r1}).
Hence we see for $z=\tilde{k}\rho$, $J_{n}(z)$ blows up exponentially, while $H_{n}^{(1)}(z)$ decreases exponentially  as $\rho\rightarrow 0^+$.
In the sequel, we let $\sqrt{a+ib}=\alpha+i\beta$ with $\alpha, \beta>0$. For $\frac{1}{C_{0}} \frac{J_{n}'(\tilde{k}\rho)}{J_{n}(\tilde{k}\rho)}$, as $\rho\rightarrow 0^+$,
by (\ref{asym:special}) we have
\begin{equation}\label{eq:1}
\frac{1}{C_{0}} \frac{J_{n}'(\tilde{k}\rho)}{J_{n}(\tilde{k}\rho)}=\frac{1}{C_0}\frac{\frac{n}{\tilde{k}\rho}J_{n}(\tilde{k}\rho) - J_{n+1}(\tilde{k}\rho)}{J_{n}(\tilde{k}\rho)} \sim  - e^{i\pi/2}\frac{1}{C_{0}} = - e^{i\pi/2}(\alpha+i\beta) \sqrt{\sigma_{l}} \rightarrow +0.
\end{equation}
For $ \Upsilon_{0} \frac{H_{n}^{(1)}(\tilde{k}\rho)}{J_{n}(\tilde{k}\rho)}$, if $\Upsilon_0$ is the fraction in (\ref{an:bn}), then we have
\begin{equation}\label{equation:1}
\begin{split}
& \Upsilon_{0} \frac{H_{n}^{(1)}(\tilde{k}\rho)}{J_{n}(\tilde{k}\rho)}
=- \frac{J_{n}'(\tilde{k}\rho/2) - C_{0} A \frac{J_{n}'(k_{2}\rho/2)}{J_{n}(k_{2}\rho/2)} J_{n}(\tilde{k}\rho/2)}
{{H_{n}^{(1)}}'(\tilde{k}\rho/2) - C_{0} A \frac{J_{n}'(k_{2}\rho/2)}{J_{n}(k_{2}\rho/2)} H_{n}^{(1)}(\tilde{k}\rho/2)}\frac{H_{n}^{(1)}(\tilde{k}\rho)}{J_{n}(\tilde{k}\rho)}\\
=&- \frac{J_{n}'(\tilde{k}\rho/2) - C_{0} A \frac{J_{n}'(k_{2}\rho/2)}{J_{n}(k_{2}\rho/2)}J_{n}(\tilde{k}\rho/2)} {J_{n}(\tilde{k}\rho)}
\frac{H_{n}^{(1)}(\tilde{k}\rho)}{{H_{n}^{(1)}}'(\tilde{k}\rho/2) - C_{0} A \frac{J_{n}'(k_{2}\rho/2)}{J_{n}(k_{2}\rho/2)}H_{n}^{(1)}(\tilde{k}\rho/2)}\\
:=& \mathcal{Y}_1\times \mathcal{Y}_2.
\end{split}
\end{equation}
By straightforward asymptotic analysis, one can show as $\rho\rightarrow 0^+$
\begin{equation}\label{eq:tem1}
\begin{split}
\left|\mathcal{Y}_1\right|= & \left|\frac{J_{n}'(\tilde{k}\rho/2) - C_{0} A \frac{J_{n}'(k_{2}\rho/2)}{J_{n}(k_{2}\rho/2)}J_{n}(\tilde{k}\rho/2)} {J_{n}(\tilde{k}\rho)}\right|\\
\sim & \left|\left(e^{i\pi/2} - C_{0} A \frac{J_{n}'(k_{2}\rho/2)}{J_{n}(k_{2}\rho/2)}\right) \frac{J_{n}(\tilde{k}\rho/2)}{J_{n}(\tilde{k}\rho)}\right|\\
\leq & \widetilde{C}\left(1+(n+1)|\alpha+i\beta|\rho^{-2-\delta}\right)\left|e^{-\frac{\beta}{2}k \rho^{-\delta}}\right|,
\end{split}
\end{equation}
where $\widetilde{C}$ is a constant independent of $\rho$. That is, $|\mathcal{Y}_1|$ decreases to $0$ more quickly than $\rho^r$ for any $r>0$. Similarly, one can show that $|\mathcal{Y}_2|$ decreases to $0$ more quickly than $\rho^r$ for any $r>0$. Furthermore, if $\Upsilon_0$ is the fraction in (\ref{eq:nn1}), by completely similar arguments one can show that $ \Upsilon_{0} \frac{H_{n}^{(1)}(\tilde{k}\rho)}{J_{n}(\tilde{k}\rho)}$ also decays more quickly than $\rho^r$ for any $r>0$. Hence, by (\ref{equation:h})--(\ref{eq:tem1}) and $\sigma_{l} = \rho^{2+2\delta}$, we have
\begin{equation}\label{eq:tem2}
\mathcal{H}(\sigma_{l}, \rho) \sim \frac{1}{C_{0}} \frac{J_{n}'(\tilde{k}\rho)}{J_{n}(\tilde{k}\rho)} \sim  - e^{i\pi/2}\frac{1}{C_{0}} = - e^{i\pi/2}(\alpha+i\beta) \rho^{1+\delta} \quad\mbox{as\ \ $\rho\rightarrow 0^+$}.
\end{equation}
Now, by (\ref{eq:tem2}) and (\ref{equation:dn}) we have
\begin{equation}
d_{n} = - \frac{i^{n}J_{n}'(k \rho) +e^{i\pi/2}(\alpha+\beta i) \rho^{1+\delta} i^{n}J_{n}(k\rho)}
{{H_{n}^{(1)}}'(k \rho) + e^{i\pi/2} (\alpha+\beta i) \rho^{1+\delta} H_{n}^{(1)}(k\rho)}.
\end{equation}
Using the asymptotic developments of Bessel functions and their derivatives (cf. \cite{LiuIMA}), we further have
\begin{equation}
\begin{cases}
d_{0} \sim - \dfrac{ -k \rho/2  + e^{i\pi/2}(\alpha+i\beta) \rho^{1+\delta} }{i\frac{2}{\pi k\rho}+  e^{i\pi/2}(\alpha+i\beta) \rho^{1+\delta} \frac{2i}{\pi}ln(k\rho/2) }, \ \ n = 0, \\
d_{n}  \sim -\dfrac{\frac{i^{n} n (k\rho)^{n-1}}{2^{n} \Gamma(n+1)} + e^{i\pi/2}(\alpha+i\beta) \rho^{1+\delta} i^{n} \frac{(k\rho)^{n}}{2^{n}\Gamma(n+1)}}
{\frac{i 2^{n} n! }{\pi (k\rho)^{n+1}} -  e^{i\pi/2}(\alpha+i\beta) \rho^{1+\delta} i \frac{2^{n}(n-1)!}{\pi (k\rho)^{n}}},  \ n \in \mathbb{N},
\end{cases}
\end{equation}
which imply
\begin{equation}\label{eq:temp4}
\begin{cases}
d_{0} \sim \mathcal{O}(\rho^{2}), \ \  n = 0 \\
d_{n} \sim \mathcal{O}(\rho^{2n}), \ \ n \in  \mathbb{N}
\end{cases}
\end{equation}
and
\begin{equation}\label{eq:temp1}
\begin{split}
&\left|d_{0} -\left[-\frac{J_{0}'(k\rho)}{{H_{0}^{(1)}}'(k\rho)}\right]\right| \leq \pi k |\alpha+i\beta| \rho^{2+\delta},\\
&\left|d_{n} -  \left[-\frac{i^{n}J_{n}'(k\rho)}{{H_{n}^{(1)}}'(k\rho)}\right]\right| \leq \frac{4\pi |\alpha+i\beta|}{k^{1+\delta}} \frac{1}{(2^{n}n!)^{2}} (k\rho)^{2n+2+\delta}.
\end{split}
\end{equation}

Since
\begin{equation}\label{eq:tem3}
u^{s}(x)=\sum_{n = -\infty}^{\infty} d_{n}H_{n}^{(1)}(k|x|)e^{in\theta},
\end{equation}
by taking $|x|\rightarrow+\infty$, together with (\ref{eq:temp4}), one has by direct calculations that the corresponding scattering amplitude satisfies
\begin{equation}
\left|\mathcal{A}(\hat{x},d)\right|\leq C\rho^2,
\end{equation}
where $C$ is a positive constant that remains uniform as $\rho\rightarrow 0^+$. That is, the construction (\ref{eq:phy1}) gives an near-cloaking device within $\rho^2$ of the ideal cloaking.

Now, we look into the series representation of the scattered wave filed (\ref{eq:tem3}). If $D_\rho$ is a sound-hard obstacle, the scattered wave corresponding to $e^{ikx\cdot d}$ is given by (see eqn. (3.19) in \cite{LiuIMA})
\begin{equation}\label{eq:sh}
\tilde{u}^{s}(x) = -\sum_{n = -\infty}^{\infty} \frac{i^{n}J_{n}'(k\rho)}{{H_{n}^{(1)}}'(k\rho)}H_{n}^{(1)}(kx)e^{in\theta}.
\end{equation}
By (\ref{eq:temp1}), (\ref{eq:tem3}) and (\ref{eq:sh}), one has by direct verifications that for $\rho$ sufficiently small
\begin{equation}\label{eq:sh2}
|u^s(x)-\tilde{u}^s(x)|\leq C\rho^{2+\delta}
\end{equation}
for any $x\in\mathbb{R}^2\backslash \bar{B}_{\epsilon_0}$ with $\epsilon_0>\rho$ a fixed constant,
where $C$ depends only on $k$ but independent of $\rho$.
We remark that the estimate in (\ref{eq:sh2}) would reduce to $C\rho^{1+\delta}$ for $x\in\mathbb{R}^2\backslash \bar{D}_\rho$. Hence, the construction (\ref{eq:phy1}) is actually a finite realization of the sound-hard lining construction (\ref{eq:partial cloaking device}).

The 3D case could be proved following similar arguments, which we shall sketch in the following. We shall make use of the same notations $u_0$, $u_2$, $u_{int}$, $\tilde{k}$ and $k_2$ etc.. We have the following series representations of the wave fields,
\begin{equation}\label{eq:3d}
\begin{split}
u_{0}(x) =& e^{ikx \cdot d}  + u^{s} (x) \\
=&\sum_{n = 0}^{\infty} \sum_{m = -n}^{n} i^{n} 4 \pi \overline{Y_{n}^{m}(d)}j_{n}(k|x|)Y_{n}^{m}(\hat{x}) +
\sum_{n =0}^{\infty} \sum_{n = -m}^{m} d_{n}^{m} h_{n}^{(1)}(k|x|)Y_{n}^{m}(\hat{x}), \\
u_{2}(x) =& \sum_{n = 0}^{\infty} \sum_{m = -n}^{n} a_{n}^{m}j_{n}(\tilde{k}|x|)Y_{n}^{m}(\hat{x}) + \sum_{n = 0}^{\infty} \sum_{m = -n}^{n} b_{n}^{m}h_{n}^{(1)}(\tilde{k}|x|)Y_{n}^{m}(\hat{x}), \\
u_{int}(x) =& \sum_{n = 0}^{\infty} \sum_{m = -n}^{n} c_{n}^{m} j_{n}^{m}(k_{2}|x|)Y_{n}^{m}(\hat{x}).
\end{split}
\end{equation}
By the transmission conditions, respectively on $\partial D_\rho$ and $\partial D_{\rho/2}$, we have
\begin{equation}\label{boudary3d:big}
\begin{cases}
a_{n}^{m}j_{n}(\tilde{k}\rho) + b_{n}^{m}h_{n}^{(1)}(\tilde{k}\rho) = i^{n} 4\pi \overline{Y_{n}^{m}(d)}j_{n}(k\rho) + d_{n}^{m} h_{n}^{(1)}(k\rho), \\
\sigma_{l}\left[\tilde{k}a_{n}^{m}j_{n}'(\tilde{k}\rho)+ \tilde{k} b_{n}^{m} {h_{n}^{(1)}}'(\tilde{k}\rho)\right] = i^{n} 4\pi k \overline{Y_{n}^{m}(d)} j_{n}'(k\rho) + k d_{n}^{m} {h_{n}^{(1)}}'(k\rho),
\end{cases}
\end{equation}
and
\begin{equation}\label{boudary3d:small}
\begin{cases}
a_{n}^{m}j_{n}(\tilde{k}\rho/2) + b_{n}^{m}h_{n}^{(1)}(\tilde{k}\rho/2) = c_{n}^{m}j_{n}(k_{2}\rho/2), \\
\sigma_{l}\left[\tilde{k}a_{n}^{m}j_{n}'(\tilde{k}\rho/2)+ \tilde{k} b_{n}^{m} {h_{n}^{(1)}}'(\tilde{k}\rho/2)\right] = \sigma_{a} c_{n}^{m} k_{2} j_{n}'(k_{2}\rho/2).
\end{cases}
\end{equation}
For this 3D case, we have $A = \sqrt{\sigma_{a}'q_{a}'}/\rho^{2}$, $k_{2} = \frac{k}{\rho}\sqrt{\frac{q_{a}'}{\sigma_{a}'}}$.
Similar to the 2D case (see (\ref{boundary1:small}), (\ref{an:bn}) and (\ref{eq:nn1})), we would solve the linear systems (\ref{boudary3d:big}) and (\ref{boudary3d:small}) and we need to distinguish two cases $j_n(k_2\rho/2)\neq 0$ and $j_n(k_2\rho/2)=0$. In the following, we only present the more complicated case with $j_n(k_2\rho/2)\neq 0$ and the other case could be handled in a completely similar manner. By solving (\ref{boudary3d:big}) and (\ref{boudary3d:small}), we have
\begin{equation}\label{eq:dnm}
d_{n}^{m} = -\frac{i^{n}4\pi \overline{Y_{n}^{m}(d)}j_{n}'(k\rho) - \frac{1}{C_{0}} \frac{j_{n}'(\tilde{k}\rho) + \Upsilon_{0} {h_{n}^{(1)}}'(\tilde{k}\rho)}{j_{n}(\tilde{k}\rho) + \Upsilon_{0} h_{n}^{(1)}(\tilde{k}\rho)}i^{n} 4\pi \overline{Y_{n}^{m}(d)}j_{n}(k\rho)} {{h_{n}^{(1)}}'(k\rho) - \frac{1}{C_{0}} \frac{j_{n}'(\tilde{k}\rho) + \Upsilon_{0} {h_{n}^{(1)}}'(\tilde{k}\rho)}{j_{n}(\tilde{k}\rho) + \Upsilon_{0} h_{n}^{(1)}(\tilde{k}\rho)} h_{n}^{(1)}(k\rho)},
\end{equation}
where
\[
\Upsilon_0:= - \frac{j_{n}'(\tilde{k}\rho/2) - C_{0}A \frac{j_{n}'(k_{2}\rho/2)}{j_{n}(k_{2}\rho/2)}j_{n}(\tilde{k}\rho/2)}{{h_{n}^{(1)}}'(\tilde{k}\rho/2) - C_{0}A \frac{j_{n}'(k_{2}\rho/2)}{j_{n}(k_{2}\rho/2)} h_{n}^{(1)}(\tilde{k}\rho/2)}.
\]
By similar asymptotic analyses as those for the 2D case, one can show that as $\rho\rightarrow 0^+$,
{
\[
 \frac{1}{C_{0}} \frac{j_{n}'(\tilde{k}\rho) + \Upsilon_{0} {h_{n}^{(1)}}'(\tilde{k}\rho)}{j_{n}(\tilde{k}\rho) + \Upsilon_{0} h_{n}^{(1)}(\tilde{k}\rho)} \rightarrow  -\frac{1}{C_{0}}e^{i\pi/2} =  -e^{i\pi/2}(\alpha+i\beta)\sqrt{\sigma_{l}}.
\]
Then by (\ref{eq:dnm}) we have
\begin{equation}
\begin{cases}
d_{0}^{0} \sim -\dfrac{- 4\pi \overline{Y_{0}^{0}(d)}k\rho/3 + 4\pi \overline{Y_{0}^{0}(d)}e^{i\pi/2}(\alpha+i\beta) \sqrt{\sigma_{l}}} { \frac{i}{(k\rho)^{2}} - e^{i\pi/2}(\alpha+i\beta)  \sqrt{\sigma_{l}} \frac{1}{k\rho}}, \\
d_{n}^{m} \sim -\dfrac{i^{n}4\pi \overline{Y_{n}^{m}(d)} \frac{2^{n}n! n (k\rho)^{n-1}}{(2n+1)!} +   i^{n}4\pi \overline{ Y_{n}^{m}(d) } e^{i\pi/2}(\alpha+i\beta) \sqrt{\sigma_{l}}\frac{2^{n}n!  (k\rho)^{n}}{(2n+1)!}} {\frac{i(2n)!(n+1)}{2^{n}n!(k\rho)^{n+2}} - e^{i\pi/2}(\alpha+i\beta) \sqrt{\sigma_{l}} \frac{i(2n)!}{2^{n}n!(k\rho)^{n+1}}},
\end{cases}
\end{equation}
}
which in turn implies
\begin{equation}
\begin{cases}
d_{0}^{0} \sim \mathcal{O}(\rho^{3}), \\
d_{n}^{m} \sim \mathcal{O}(\rho^{2n+1}),
\end{cases}
\end{equation}
and
\begin{equation}\label{coef:diff}
\begin{split}
&\left|d_{0}^{0} - \left[-\frac{4\pi \overline{Y_{0}^{0}(d)}j_{0}'(k\rho)}{{h_{0}^{(1)}}'(k\rho)}\right]\right| \leq 8\pi k^{2} |\alpha+i\beta| |Y_{0}^{0}(d)|\rho^{3+\delta}, \\
& \left|d_{n}^{m} - \left[-\frac{i^{n}4\pi \overline{Y_{n}^{m}(d)}j_{n}'(k\rho)}{{h_{n}^{(1)}}'(k\rho)}\right]\right| \leq \frac{ 4(2^{n}n!)^{2} {4\pi |Y_{n}^{m}(d)|}|\alpha+i\beta|k^{-1-\delta}}{(2n)!(2n+1)!(n+1)}(k\rho)^{2n+3+\delta}.
\end{split}
\end{equation}
With the above preparations, one can show that the corresponding scattering amplitude corresponding to the cloaking device (\ref{eq:phy1}) in $\mathbb{R}^3$ satisfies
\begin{equation}\label{eq:111}
|\mathcal{A}(\hat{x},d)|\leq C\rho^3,
\end{equation}
where $C$ remains uniform as $\rho\rightarrow 0^+$. That is, (\ref{eq:phy1}) gives a near-cloaking device within $\rho^3$ of the ideal cloaking. Furthermore, by comparing the scattered wave field $u^s$ in (\ref{eq:3d}) to that of a 3D sound-hard ball $B_\rho$ (cf. \cite{LiuIMA}), together with (\ref{coef:diff}), one can show that the deviation between them is within $\rho^{3+\delta}$ in $\mathbb{R}^3\backslash\bar{B}_{\epsilon_0}$ for any fixed $\epsilon_0>\rho$, and within $\rho^{2+\delta}$ in $\mathbb{R}^3\backslash\bar{B}_\rho$. Hence again, we come to the conclusion that the construction (\ref{eq:phy1}) is a finite realization of the sound-hard lining construction (\ref{eq:partial cloaking device}).

In summary, we have shown in this section that
\begin{thm}\label{thm:2}
Let $D_\rho=B_\rho$ and $\{D_\rho\backslash\bar{D}_{\rho/2};\sigma_l, q_l\}$ be given by (\ref{eq:lossy para}), and $\{D_{\rho/2};$ $\sigma_a,q_a\}$ be arbitrary but uniform. Then the construction (\ref{eq:phy1}) produces a near-cloaking device within $\rho^N$ of ideal cloaking. Furthermore, (\ref{eq:phy1}) is a finite realization of (\ref{eq:partial cloaking device}) in the sense that the scattered wave fields corresponding to (\ref{eq:phy1}) and (\ref{eq:partial cloaking device}), respectively, deviate within $\rho^{N+\delta}$ outside the cloaking device.
\end{thm}

As we remarked earlier that Theorem~\ref{thm:2} equally holds for the general case with general geometry and arbitrary cloaked contents (see \cite{LiLiuSun}). Actually, in the subsequent section, our numerical examples are presented for variable cloaked contents.

\section{Some discussion on different cloaking schemes}

As we discussed earlier in Section~\ref{sec:intro}, two different
cloaking schemes were developed in \cite{KOVW} and \cite{Liu}. One
of the key ingredients is to introduce a special layer between the
cloaked region and the cloaking region, which are respectively
$D_{1/2}$ and $\Omega\backslash\bar{D}$ in the physical space
described in (\ref{eq:partial cloaking device}). They correspond,
respectively, $D_{\rho/2}$ and $\Omega\backslash\bar{D}_\rho$ in the
virtual space. In \cite{KOVW}, the authors introduce a lossy layer
occupying $D\backslash\bar{D}_{1/2}$ and they showed that the lossy
layer is indispensable to achieve successful near-cloak. In virtual
space, the lossy layer in \cite{KOVW} is given as
\begin{equation}\label{eq:finite}
\left\{D_{\rho}\backslash\bar{D}_{\rho/2}; I, 1+i\beta\right\},\ \ \ \beta\sim \rho^{-2}.
\end{equation}
If one lets the loss parameter $\beta$ go to infinity, the limit corresponds to the imposition of a homogeneous Dirichlet boundary condition, which is the one considered in \cite{Liu}. It is interesting to mention that the improvement of cloaking performance by imposing specific boundary condition on the cloaking interface is also considered in \cite{GKLU0}. In this context, imposition of a homogeneous Dirichlet boundary condition amounts to the lining of a layer of sound-soft material in
$D_{\rho}\backslash\bar{D}_{\rho/2}$. In this sense, the lossy layer (\ref{eq:finite}) is a finite realization
of the sound-soft layer lining. We shall refer the construction in \cite{Liu} with a sound-soft layer as {\it SS construction}, and the one in \cite{KOVW} with a finite realization of sound-soft layer as {\it FSS construction}.

In our cloaking construction (\ref{eq:partial cloaking device}), the inclusion of a sound-hard layer will significantly enhance the cloaking performance, and as specified in the Introduction we call this scheme an SH construction. In order to achieve a finite realization of the sound-hard layer, we utilize the following lossy layer
\begin{equation}\label{eq:finite 2}
\left\{D_{\rho}\backslash\bar{D}_{\rho/2}; \sigma_l, a+ib\right\},\ \ \ \sigma_l\sim\rho^{2+\delta}I,\ \ a\sim 1,\ b\sim 1.
\end{equation}
Clearly, the lossy layer (\ref{eq:finite 2}) is of significant physical and mathematical nature from (\ref{eq:finite}), and it could produce significantly improved near-cloaking performances. This losses scheme is called an FSH construction.

\section{Numerical examples}

In this section, we present some numerical experiments to
demonstrate the theoretical results established in the previous
sections. There are totally four near-cloak schemes investigated,
namely \emph{SS} (sound-soft lining), \emph{FSS} (finite sound-soft
lining), \emph{SH} (sound-hard lining) and \emph{FSH} (finite
sound-hard lining).

First, all the numerical experiments are conducted in $\mathbb{R}^2$ and
we choose $D$ and $D_\rho$ as $B_{R_1}$ and $B_\rho$,
respectively. Some key parameters are set as follows: $R_1=2$,
$R_2=3$ and $R_3=4$, thus $R_0=R_1/2=1$. Experimental settings are
shown in Figure~\ref{fig:setting}. For the setting with respect to
the scattering measurement, a Cartesian PML layer of width $1$ is
attached to the square $(-5,5)^2$ to truncate the whole space into a
finite domain with scattering boundary conditions enforced on the
outer boundary.
$(-5,5)^2\setminus \overline{B_{R_3}}$ and
$B_{R_3}\setminus\overline{B_{R_2}}$  are homogeneous surrounding
media. $B_{R_2}\setminus \overline{B_{R_1}}$ is the cloaking medium.
$B_{R_1}\setminus \overline{B_{R_0}}$, $B_{R_0}$ are the lossy layer
and the cloked region, respectively. The regularization parameter
$\rho$ ranges from $0.5$ to $10^{-5}$. We fix the wave number $k=2$,
and the incident direction $d=(1,0)^T$.

In addition, for Scheme \emph{FSS}, the lossy parameters are chosen
according to \cite{KOVW}, i.e., $\sigma_l = I$, $q_l = \rho^2 (1+2.5
\rho^{-2} i  )$ in $B_{R_2}\setminus \overline{B_{R_1}}$. For Scheme
\emph{FSH}, the lossy parameters are chosen by \eqref{eq:phy1},
\eqref{eq:lossy para} and the definition of the map
\eqref{eq:F:ball:map}, namely $\sigma_l = \rho^{2.5}I$, $q_l =
\rho^2 (3+2 i)$ in the cloaking medium of the physical space, namely
$C=1$, $\delta=0.5$, $a=3$ and $b=2$.

For Schemes \emph{FSS} and \emph{FSH}, We choose the medium
coefficients in the cloaked region
to be variable functions $\sigma_a = 1+0.3\ x^2 y^2$, $q_a=5+x$. For Scheme \emph{SH}, the
incident wave is shifted to the right by a tenth of the wave length
such that the origin is not in the valley nor peak and the incident
wave does not vanishes at the origin.

\begin{figure}
\hfill{}
\includegraphics[width=0.45\textwidth]{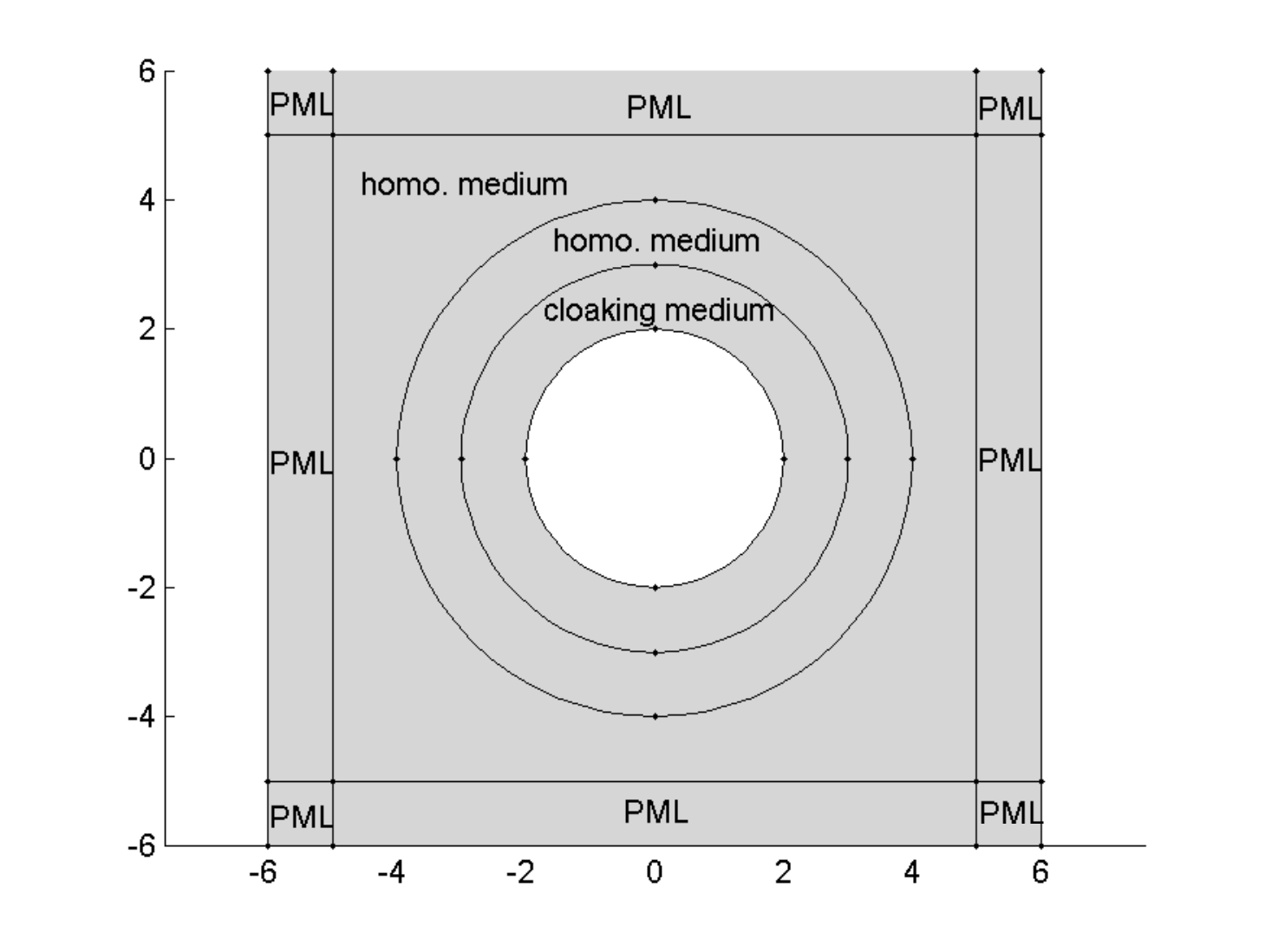}\hfill{}
%
\includegraphics[width=0.45\textwidth]{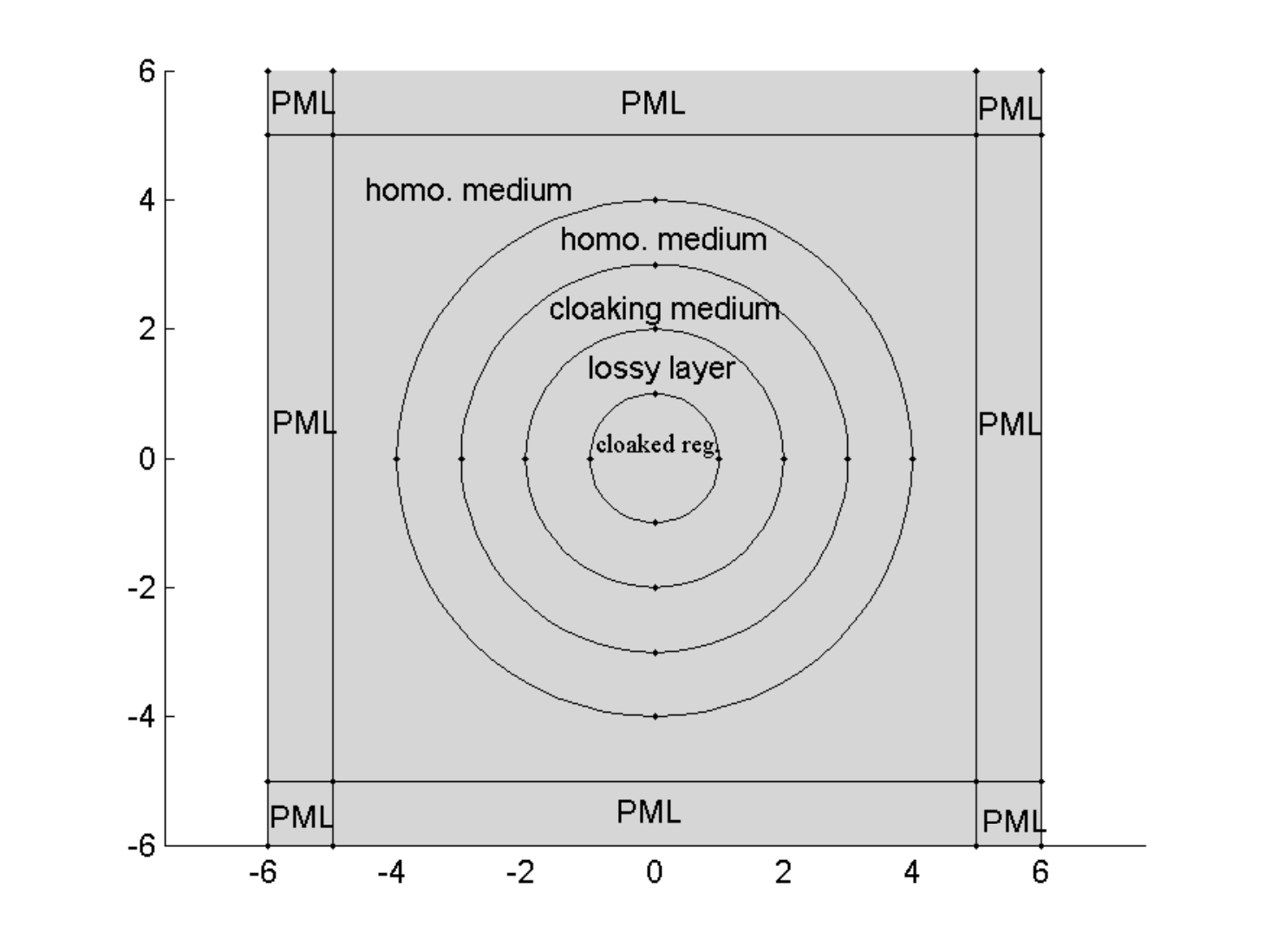}\hfill{}

\caption{\label{fig:setting} Experimental settings for scattering
measurement tests.  Left: Tests with \emph{SS} and \emph{SH} lining. Right: Tests with \emph{FSS} and \emph{FSH} lining.}
\end{figure}

From the transformation map $F$ defined in \eqref{eq:F:ball:map}, we
can determine the cloaking medium coefficients by \eqref{eq:phy1}
for any $\rho<R_1$.
It is emphasized that as $\rho$ tends to zero, the singularity of
$\sigma$ increases very fast. 
For example, it is observed that $\sigma_{11}$ grows asymptotically
as $1 / \rho$ and exhibits a thin layer of large values except some
constant background as $\rho$ decreases, which requires many local
refinements for better resolution and accounts for the boundary
layer in the finite element solutions.

%
%
%
%
%
%
%
%

%
%

The scattering amplitude, or the far field data, can be computed in
the following way (cf. \cite{ColKre}). We can represent the
scattered wave in the boundary integral form
\begin{equation}\label{eq:us}
u^{s}(x)=\int_{\partial
B_{R_3}}u^{s}(y)\frac{\partial\Phi(x,y)}{\partial\nu(y)}-\Phi(x,y)\frac{\partial
u^{s}(y)}{\partial\nu(y)}\, ds(y),\quad
x\in\mathbb{R}^{2}\setminus\overline{B_{R_3}}.
\end{equation}
Making use of the asymptotic expansion of the fundamental solution
$\Phi(x,y)=i H_0^{(1)}(k|x-y|)/4$, we then have
\begin{equation}\label{eq:uinf}
\mathcal{A}(\hat{x}):=\frac{\exp(-i
\pi/4)}{\sqrt{8k\pi}}\int_{\partial B_{R_3}}u^{s}(y)\frac{\partial
e^{-i  k\hat{x}\cdot y}}{\partial\nu(y)}-e^{-\imath k\hat{x}\cdot
y}\frac{\partial u^{s}(y)}{\partial\nu(y)}\, ds(y),
\end{equation}
where $\hat{x}=x/|x|$ is the observation direction and $\nu(y)$ is
the outward unit normal vector pointing to the infinity. That is,
the far field data can be approximated by the numerical quadrature
of \eqref{eq:uinf} using the Cauchy data of the scattered wave $u^s$
on the boundary of $B_3$.

Two groups of tests are carried out in the following.


\subsection{Near-cloaks of scheme \emph{SS} and scheme \emph{FSS}}

It is pointed out that as $\rho$ tends to zero, the medium
coefficients exhibit more and more singularity close to $\partial
B_{R_1}$. Thus finite element solutions require more degrees of
freedom to resolve the boundary layer close to $\partial B_{R_1}$.
Either the sound-soft boundary condition or the \emph{FSS} lossy
layer enforces a strong influence to the local behavior of the
transmitted wave close to $\partial B_{R_1}$, which amounts to
introducing a point source in the virtual space. Such \emph{virtual
point source} makes the boundary data show a very slow convergence
as we decrease $\rho$. To tackle with the difficulty of the boundary
layer  amounts to more than two millions of DOF's to achieve the
desired relative error tolerance when $\rho=10^{-5}$. The same
observation holds for the other schemes. To avoid the influence of the
finite element discretization error, all the following tests are
carried out on the finest mesh to show the sharpness of the
theoretical error bounds in terms of $\rho$.


The scattered wave and the  transmitted wave are plotted for Schemes
\emph{ SS} and \emph{FSS}, respectively,  for different $\rho$'s in
Figures~\ref{fig:S:us:scheme:12}.

%

\begin{figure}[htbp]

%

\hfill{}
\includegraphics[clip,width=0.32\textwidth]{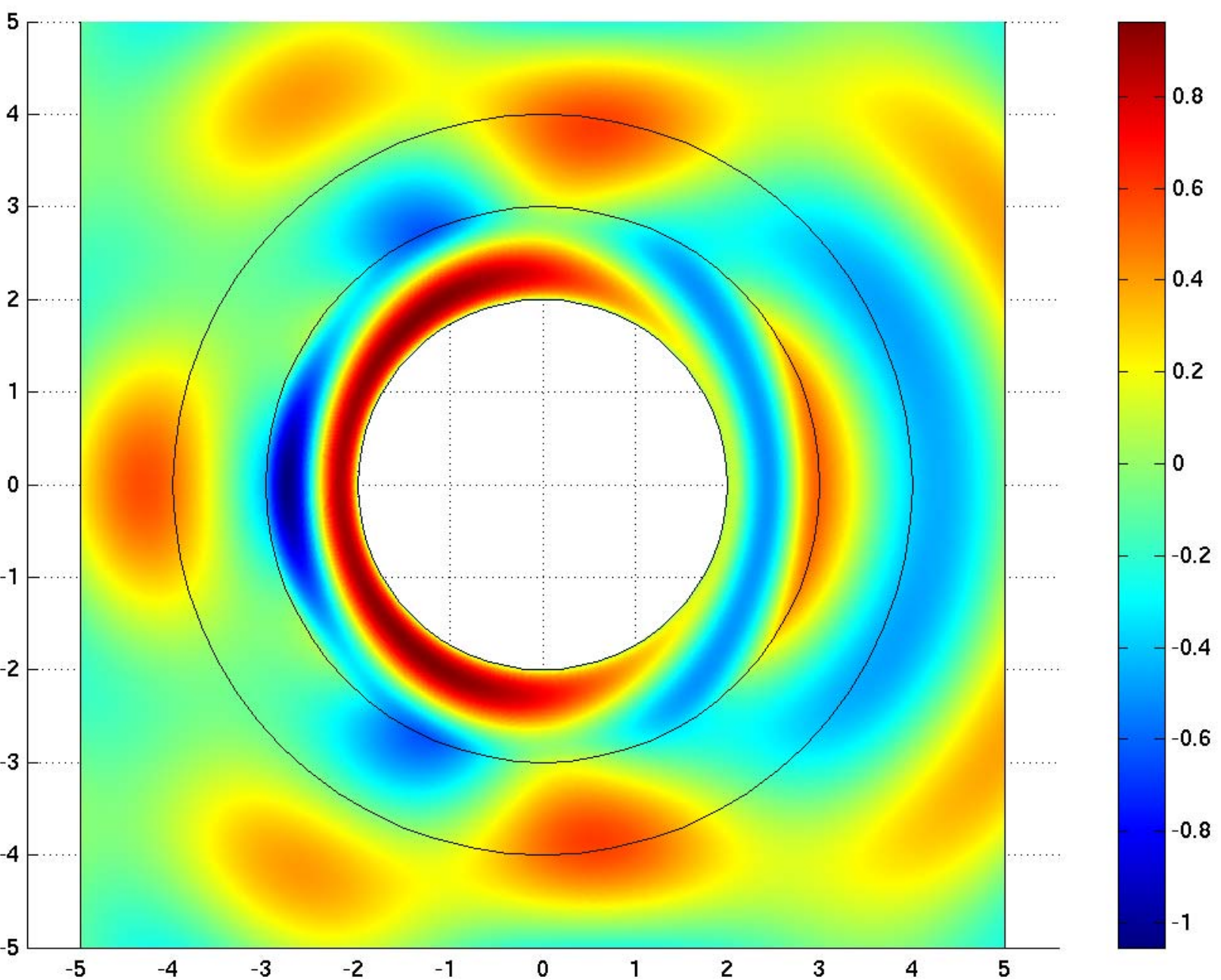}\hfill{}
\includegraphics[width=0.32\textwidth]{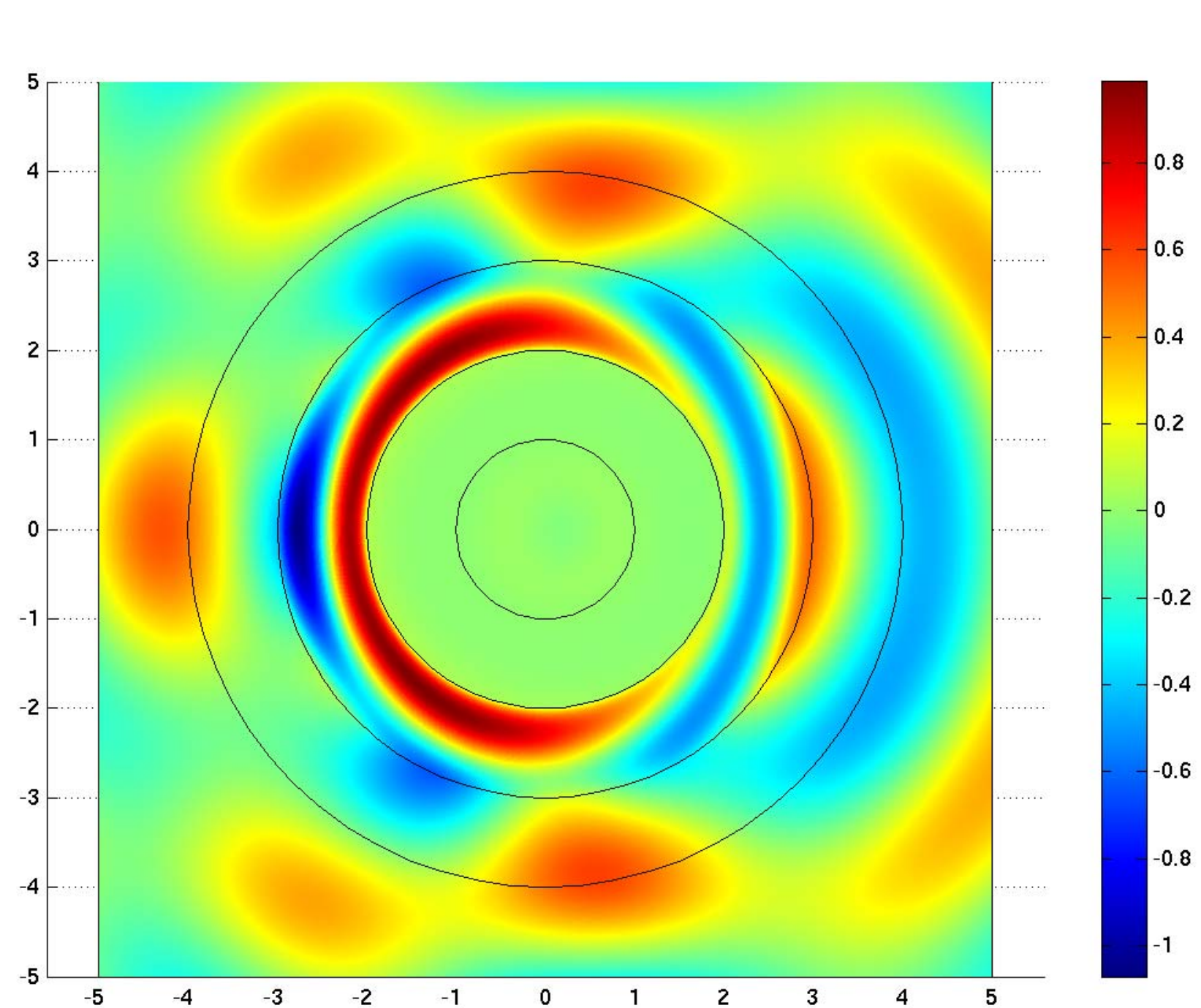}\hfill{}

\hfill{}
\includegraphics[clip,width=0.32\textwidth]{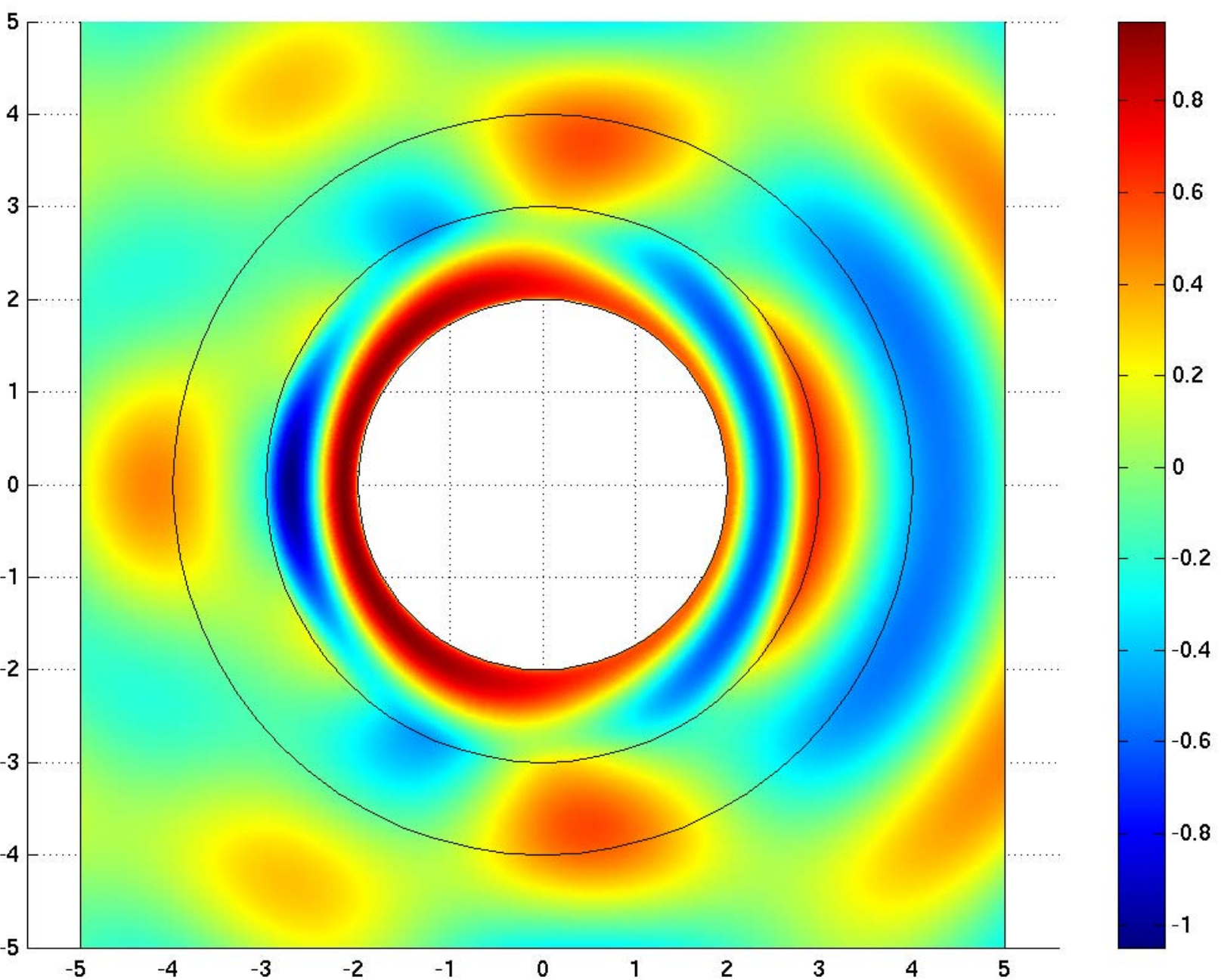}\hfill{}
\includegraphics[width=0.32\textwidth]{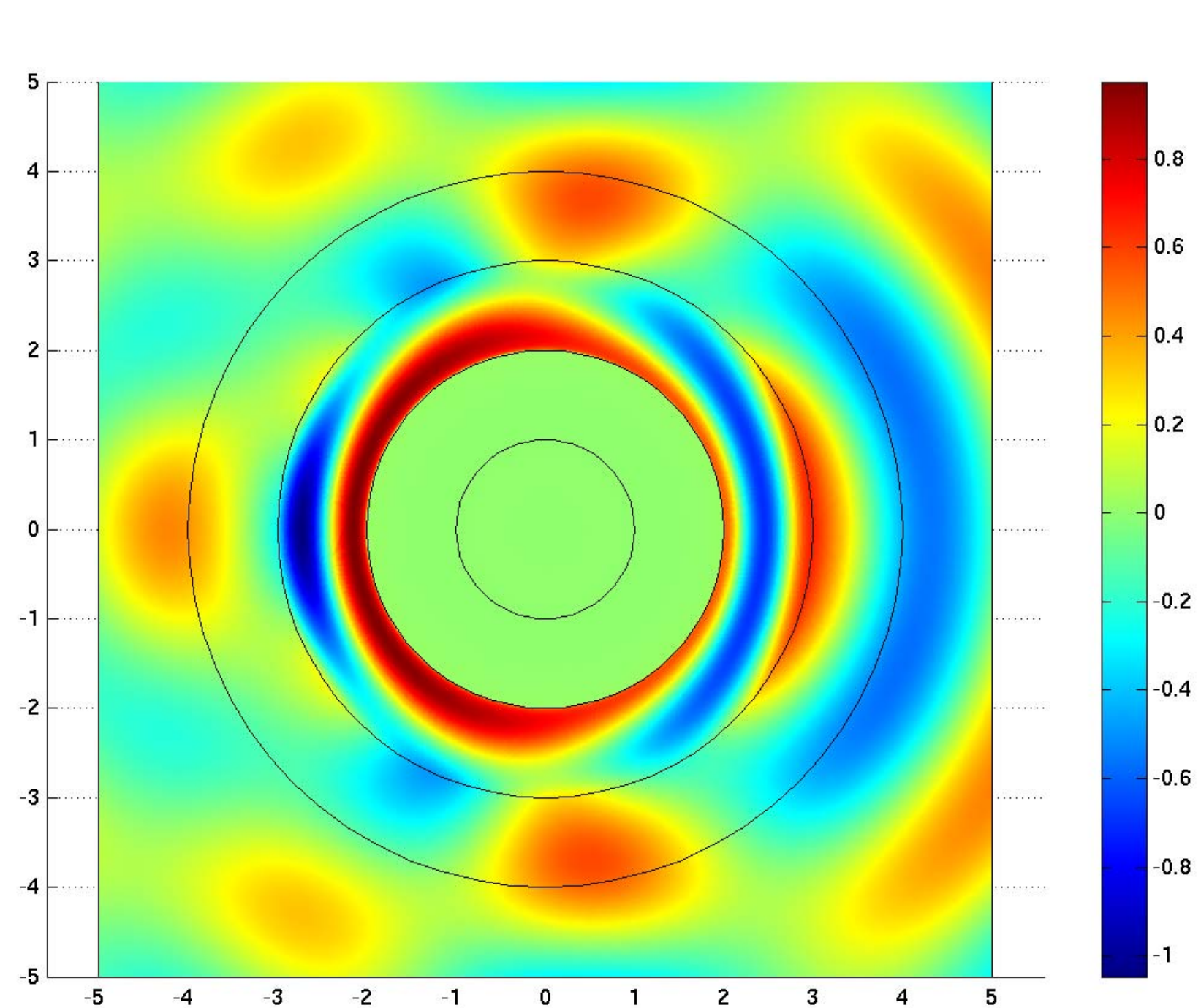}\hfill{}

\caption{\label{fig:S:us:scheme:12} The scattered wave
 and the transmitted wave (real part) with respect to $\rho=10^{-1}$ and
$10^{-5}$ from top to bottom, respectively. Left: Scheme \emph{ SS};
Right: Scheme \emph{ FSS}.}
\end{figure}

The scattering measurement data are generated according to
\eqref{eq:uinf} by solving the Helmholtz system on the truncated
domain with a PML layer. In the discrete sense, all the norms are
equivalent. So the discrete maximum norm are used to measure the
decay rate of the far field data $\mathcal{A}(\hat{x})$, namely the maximum of
the modulus of $\mathcal{A}(\hat{x})$ are taken over 100 equidistant observation
direction on $\mathbb{S}^1$. We investigate the convergence rate by
testing $\rho=1/2^j$, $j=1,2,\ldots,7$ for Schemes \emph{SS} and
\emph{FSS}. By plotting in
Figure~\ref{fig:linear:regression:scattering:SS:FSS} the convergence
history of the discrete maximum norm of $\mathcal{A}(\hat{x})$ over
$\mathbb{S}^1$ with respect to the upper bound $1/|\log_{10} \rho|$,
we see clearly that Schemes \emph{SS} and \emph{FSS} schemes indeed
achieve near-cloak, but we see that the convergence curve
demonstrates linearity asymptotically except the first few outliers
to the left of the plots in
Figure~\ref{fig:linear:regression:scattering:SS:FSS}, which verifies
the sharpness of the upper bound established in \cite{KOVW}.  But
the convergence slows down significantly as $\rho$ decreases, which
is reflected by the clustering of the blue star points in the curve.
More importantly, it can be seen from
Figure~\ref{fig:linear:regression:scattering:SS:FSS} that the
difference between the discrete maximum norms for Schemes \emph{SS}
and \emph{FSS} gets smaller as $\rho$ decreases because the
\emph{FSS} lossy layer tends effectively to the sound-soft boundary
condition for small $\rho$ and thus Scheme \emph{FSS} approaches
Scheme \emph{SS} in the limit sense as $\rho\to 0$.

\begin{figure}

\hfill{}
\includegraphics[clip,width=0.4\textwidth]{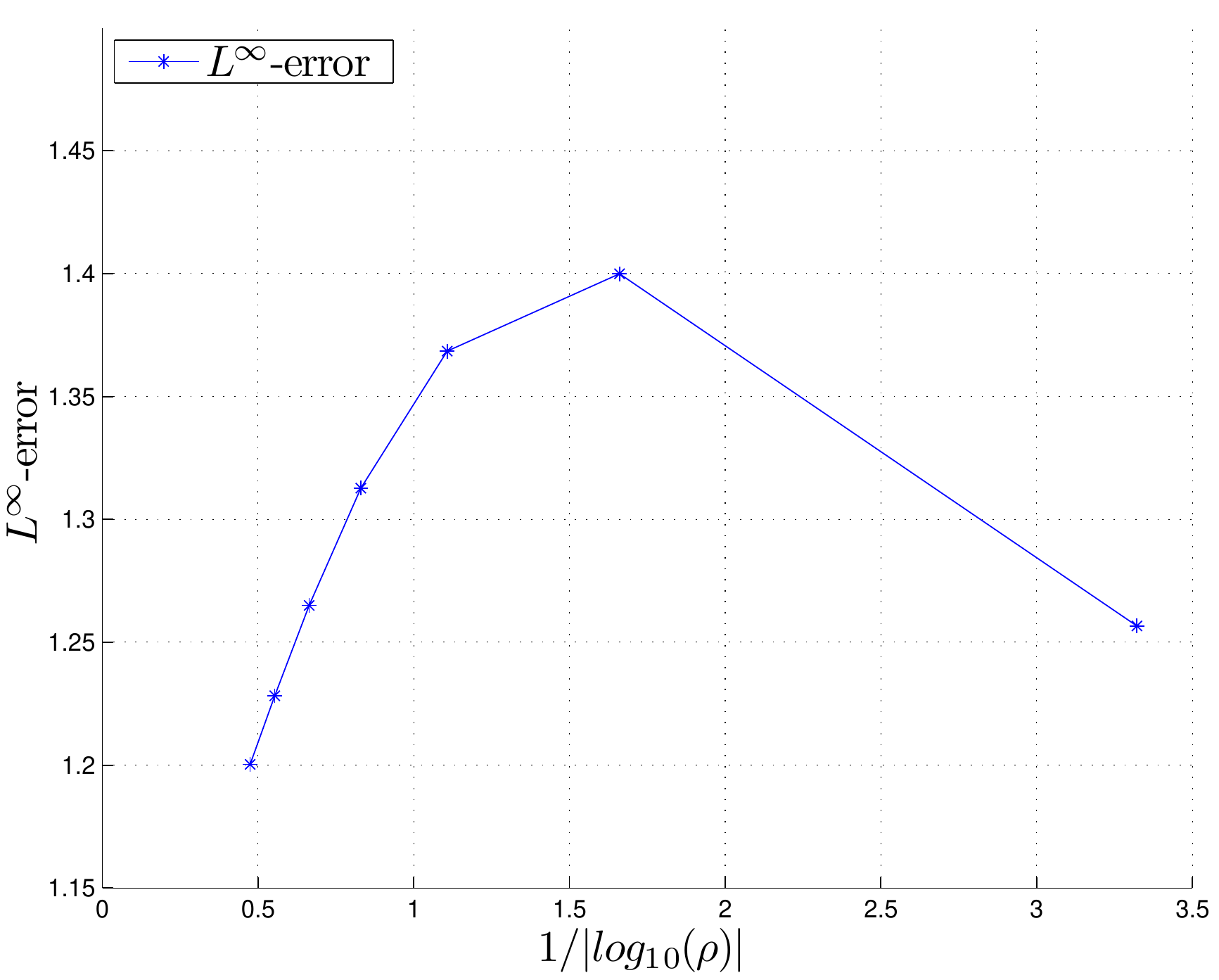}\hfill{}
\includegraphics[clip,width=0.4\textwidth]{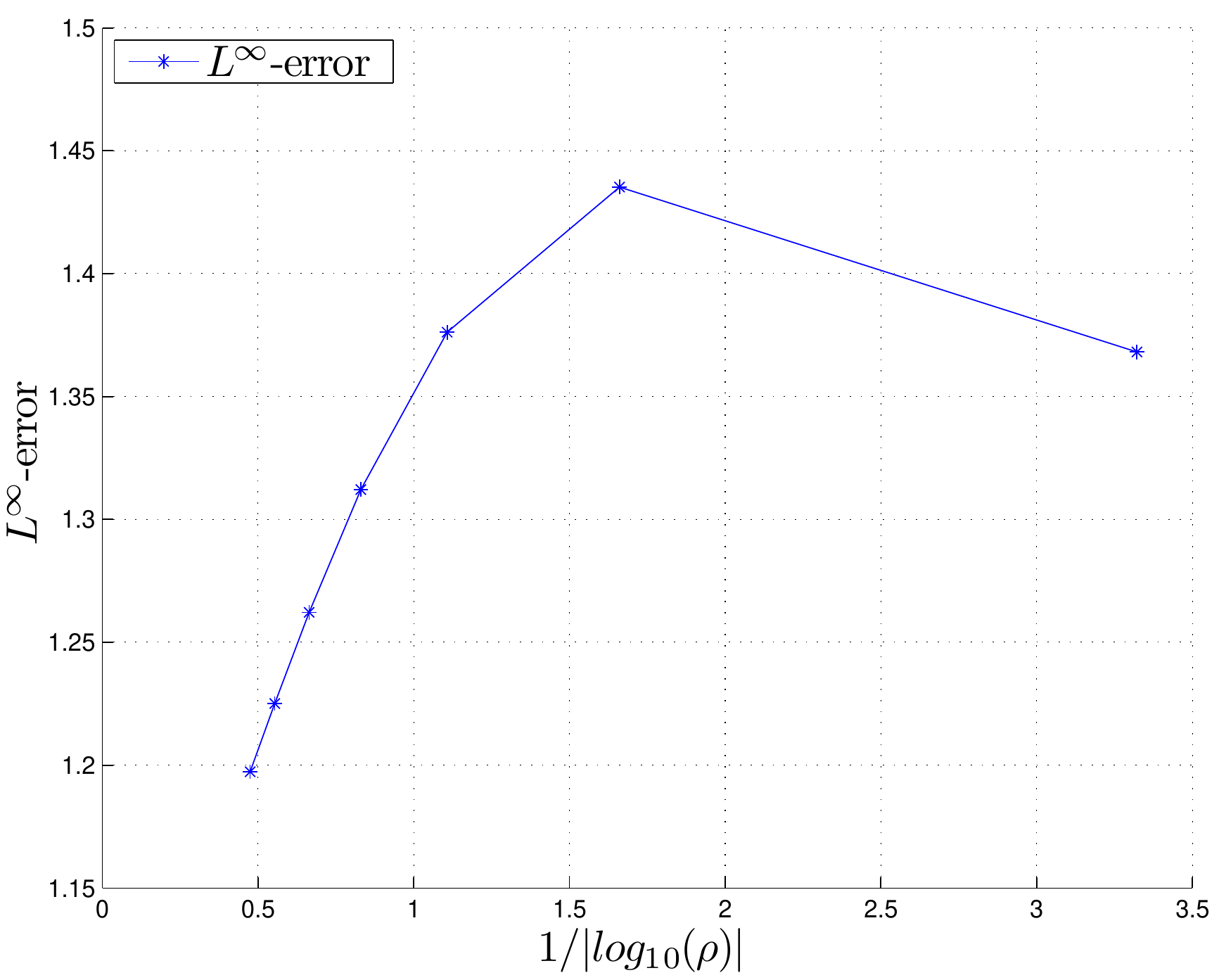}\hfill{}

%


\caption{\label{fig:linear:regression:scattering:SS:FSS} Convergence
history of scattering measurement data versus $1/|\log_{10} \rho|$
for Scheme \emph{SS} (Left) and Scheme \emph{FSS} (Right). }
\end{figure}


\subsection{Near-cloaks of scheme \emph{SH} and scheme \emph{FSH}}

Contrary to Schemes \emph{SS} and \emph{FSS}, our new schemes by
\emph{SH} and \emph{FSH} lining can achieve significantly better
near-cloak performance.

Figure~\ref{fig:S:us:scheme:3} shows the scattered wave and the
transmitted wave for {Scheme} \emph{SH} for different $\rho$'s.
Compared with Scheme  \emph{SS}, the scattered wave of Scheme
\emph{SH} decays significantly as $\rho$ decreases. Moreover, the
transmitted wave approaches more and more like a deformed plane
incident wave, the interior value of the transmitted wave tends to a
constant as $\rho\to 0$.

\begin{figure}
%

\hfill{}
\includegraphics[width=0.32\textwidth]{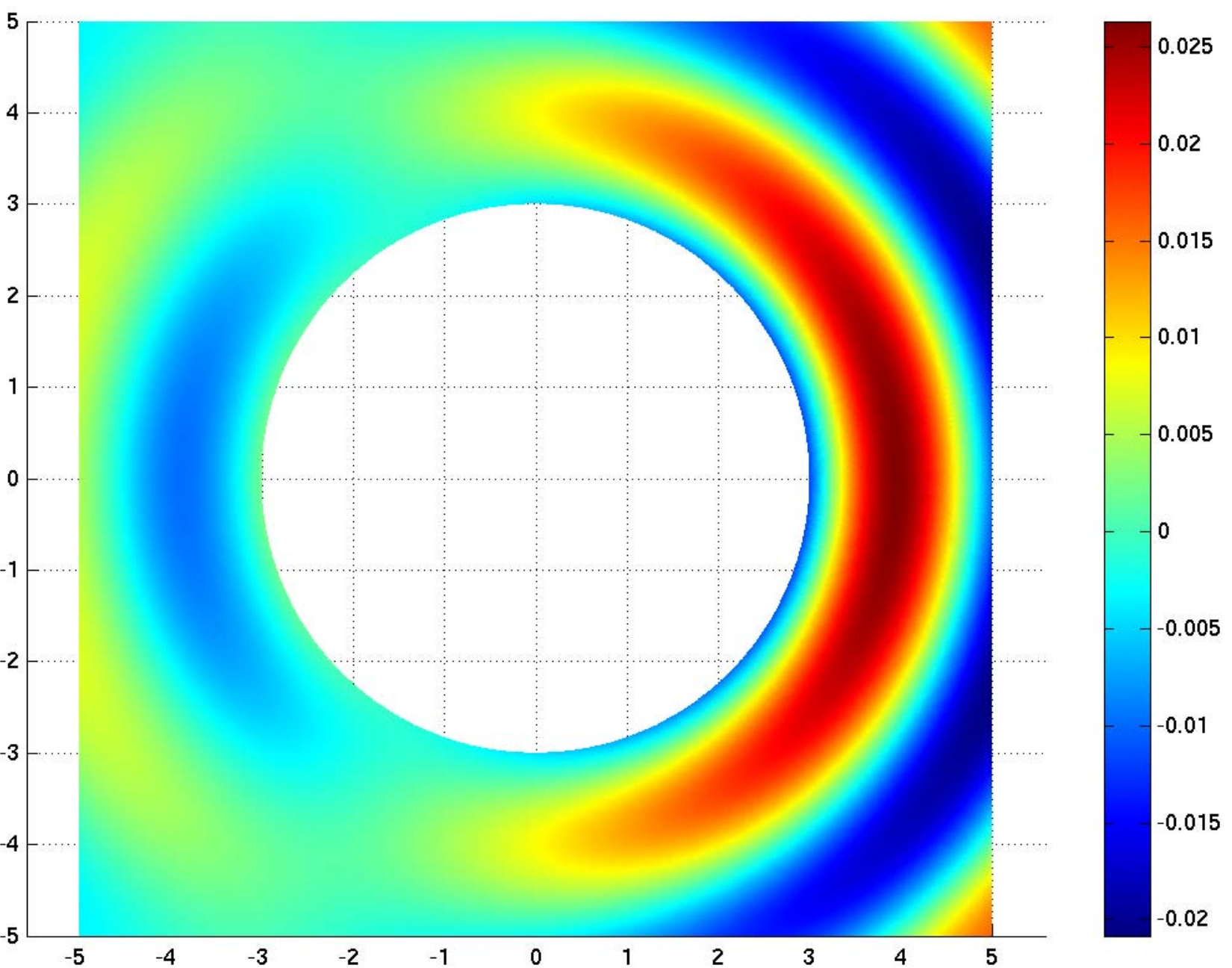}\hfill{}
\includegraphics[width=0.32\textwidth]{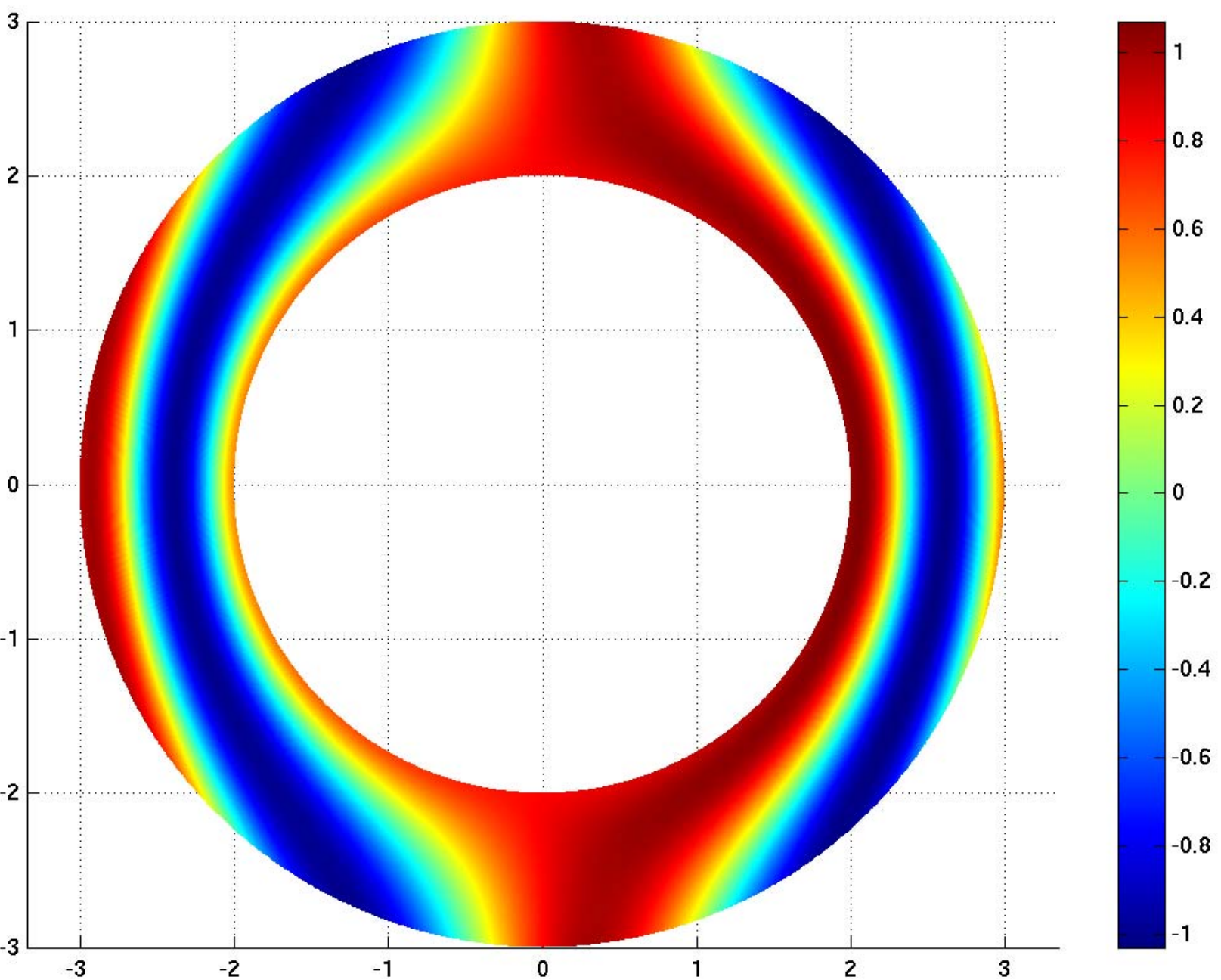}\hfill{}

\hfill{}
\includegraphics[width=0.32\textwidth]{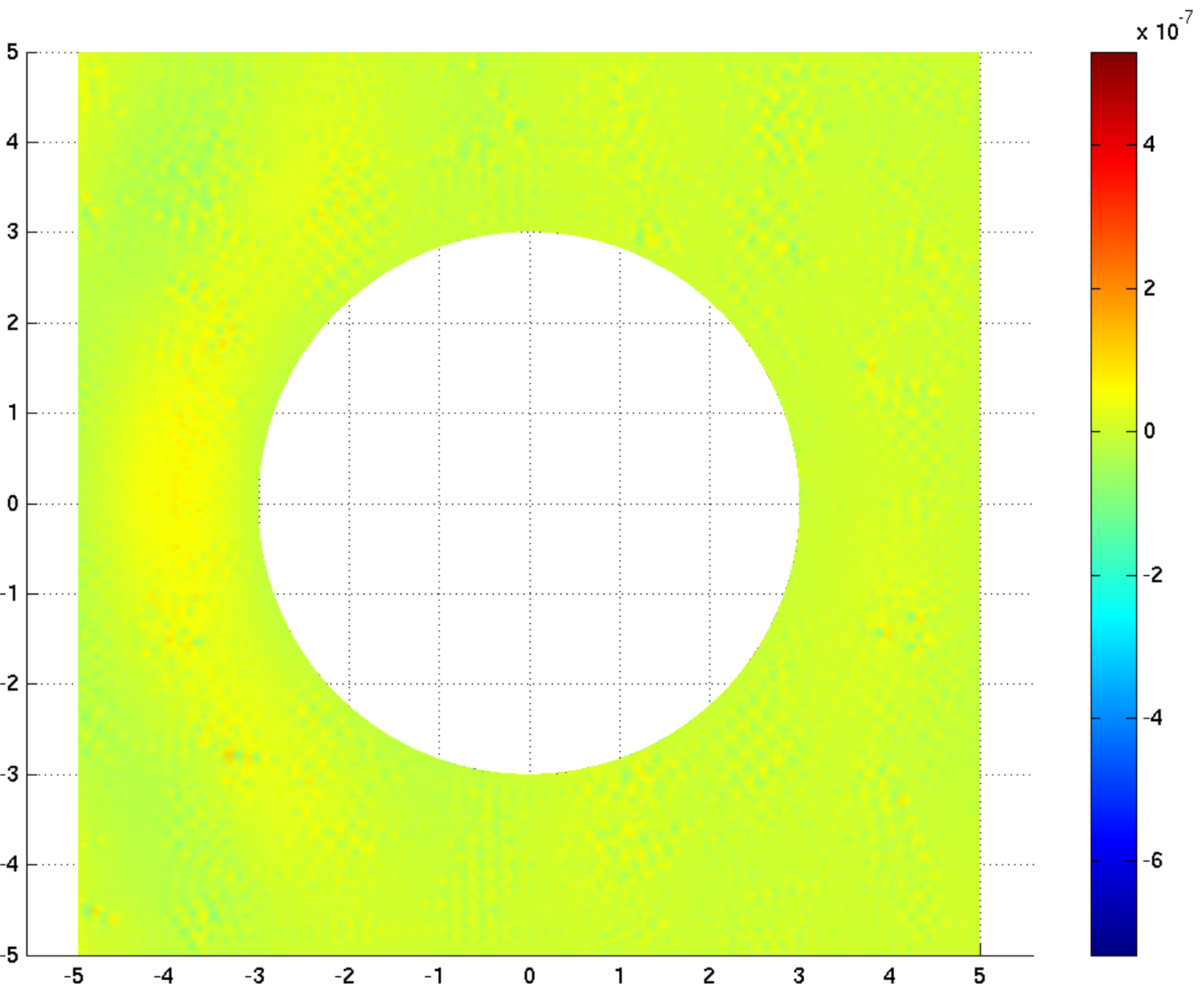}\hfill{}
\includegraphics[width=0.32\textwidth]{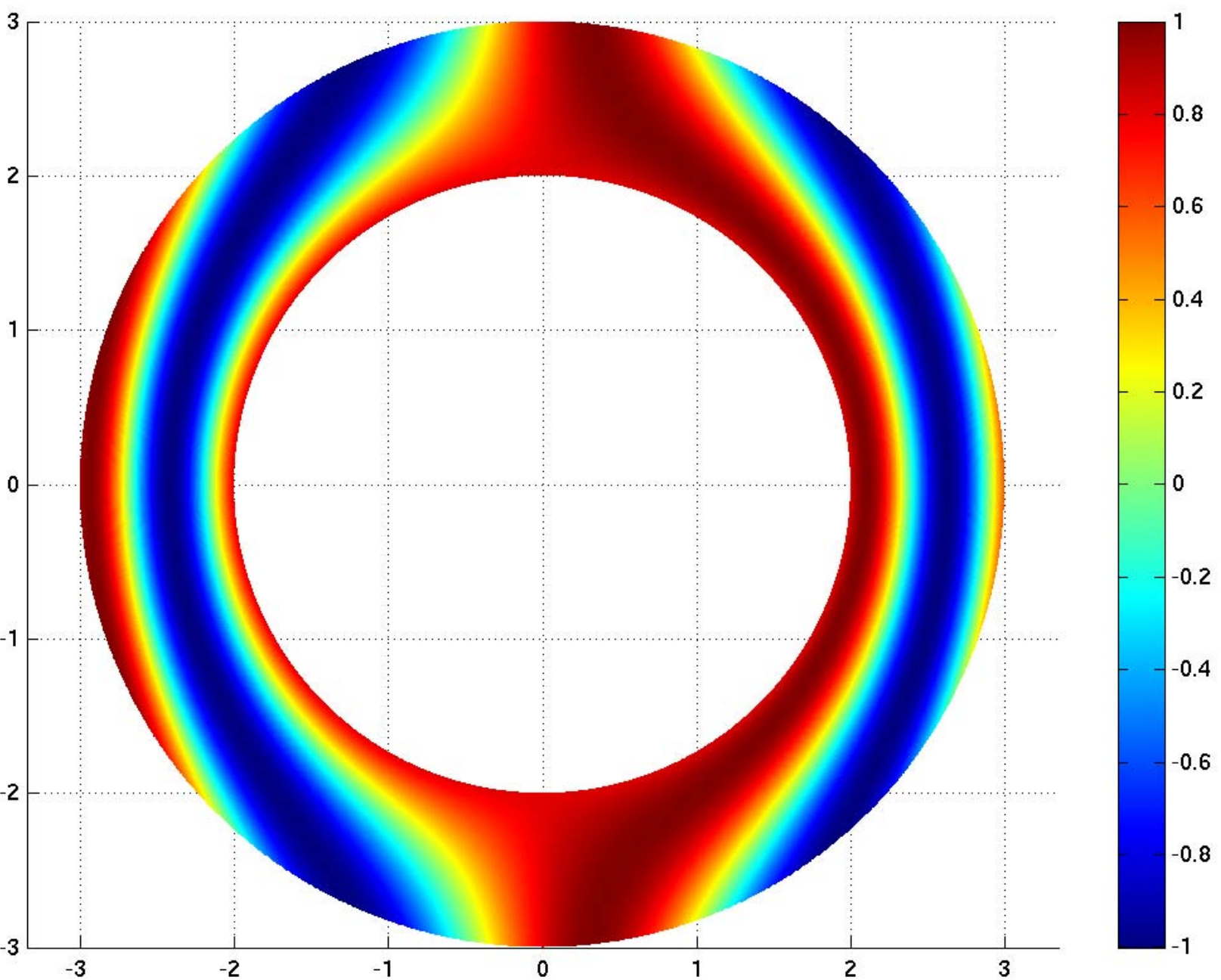}\hfill{}

\caption{\label{fig:S:us:scheme:3} The scattered wave $u^s$ (real
part) and the transmitted wave $u^t$  with respect to $\rho=10^{-1}$
and $10^{-5}$ from top to bottom, respectively, for \emph{Scheme
SH}.}
\end{figure}

We plot the scattered wave and the transmitted wave for {Scheme}
\emph{FSH} for different $\rho$'s in Figure~\ref{fig:S:us:scheme:4}.
We can observe a thin red layer near the outer boundary of the lossy
layer, on which both the acoustic potential and the normal flux
matches with those from the cloaking medium. Compared with Scheme
 \emph{FSS}, the scattered wave of Scheme  \emph{FSH} decays
significantly as $\rho$ decreases. The transmitted wave approaches
more and more like a deformed plane incident wave, and the interior
acoustic potential of the transmitted wave within the cloaked region
vanishes as $\rho\to 0$. In other words, the medium scatterer in the
cloaked region is approximately isolated from the exterior world and
has negligible affect on the scattering measurement, which means
that we achieve significant  nearly-invisibility by \emph{SH} and
\emph{FSH} constructions.

\begin{figure}
%

\hfill{}
\includegraphics[width=0.32\textwidth]{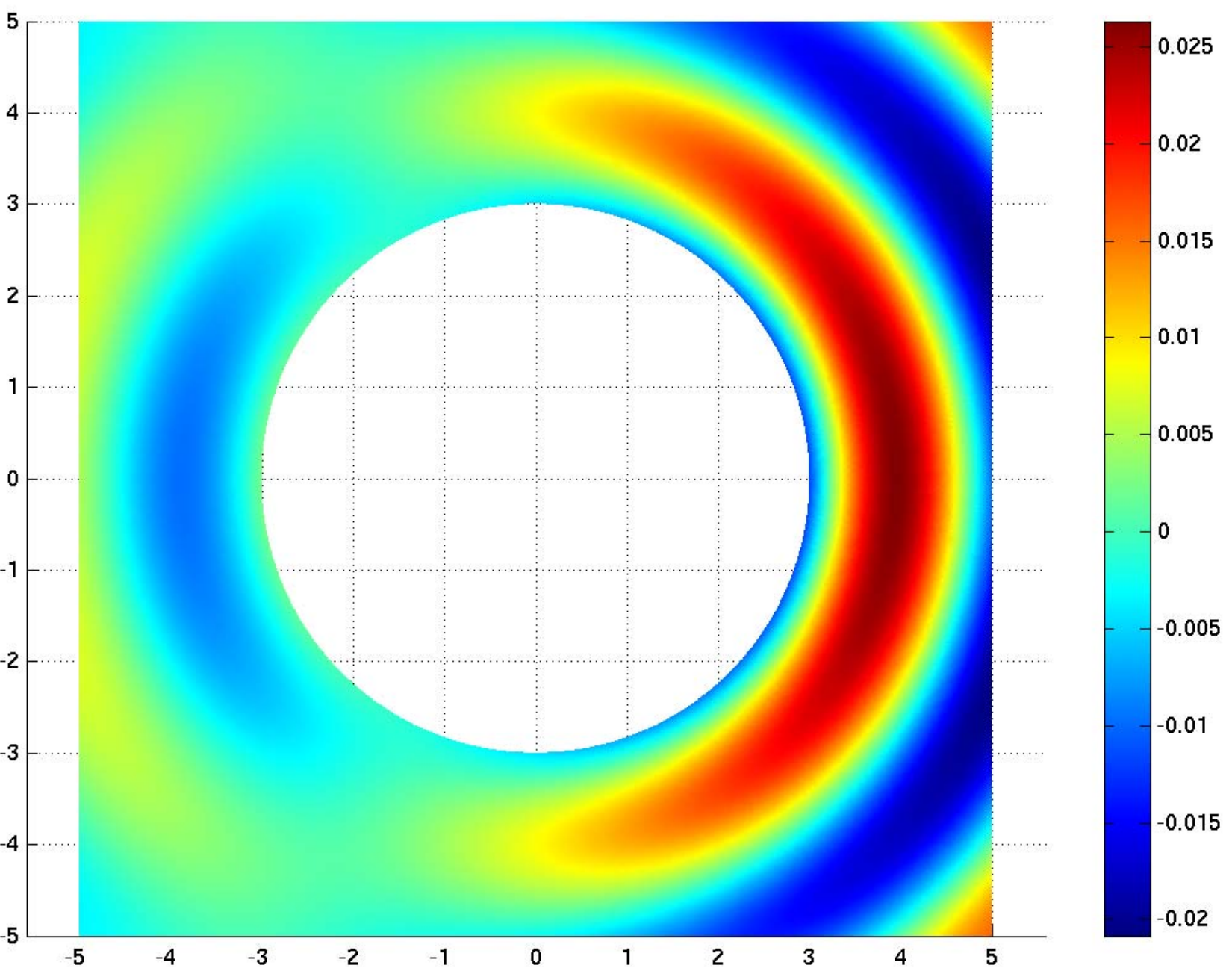}\hfill{}
\includegraphics[width=0.32\textwidth]{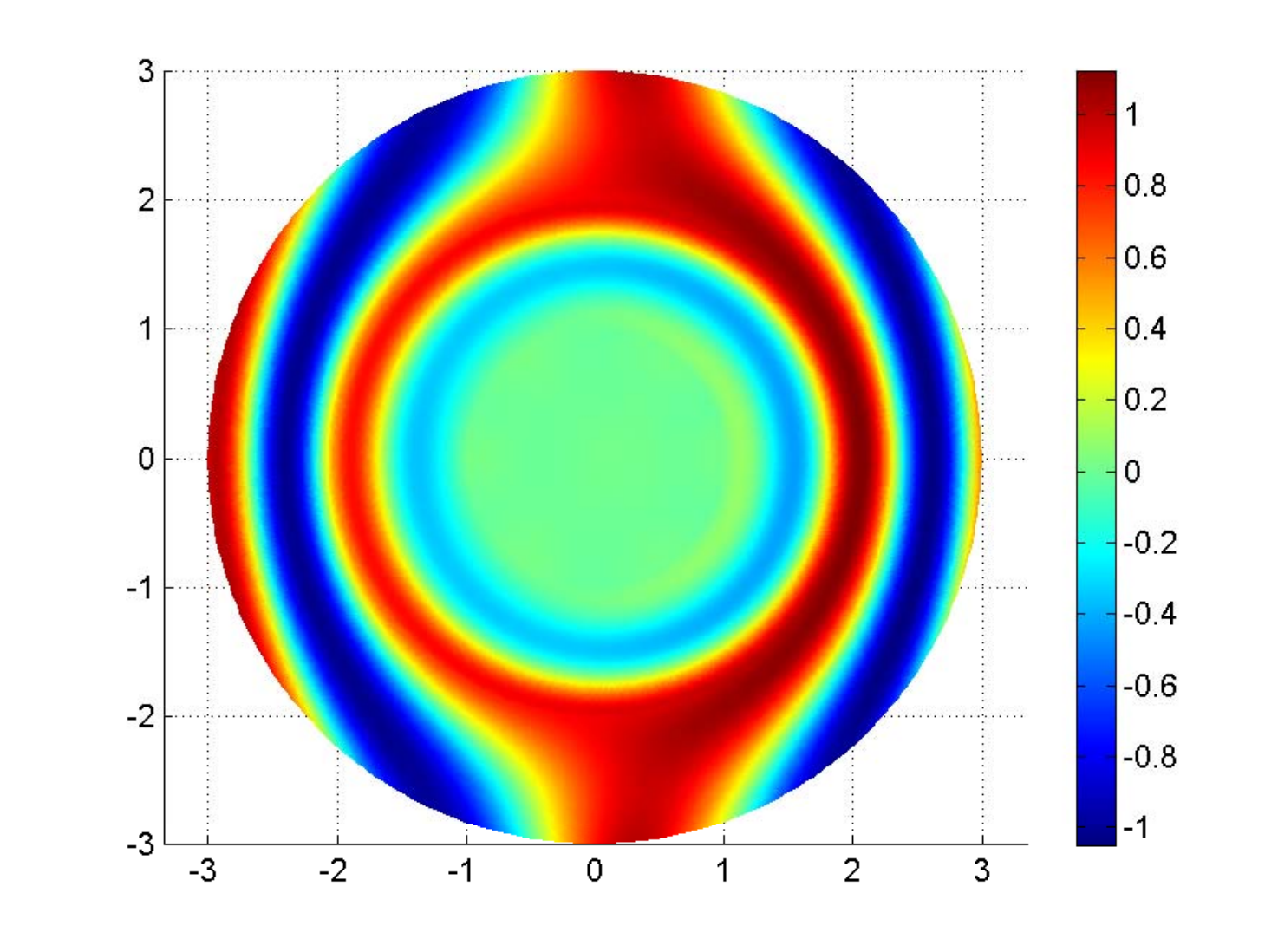}\hfill{}

\hfill{}
\includegraphics[width=0.32\textwidth]{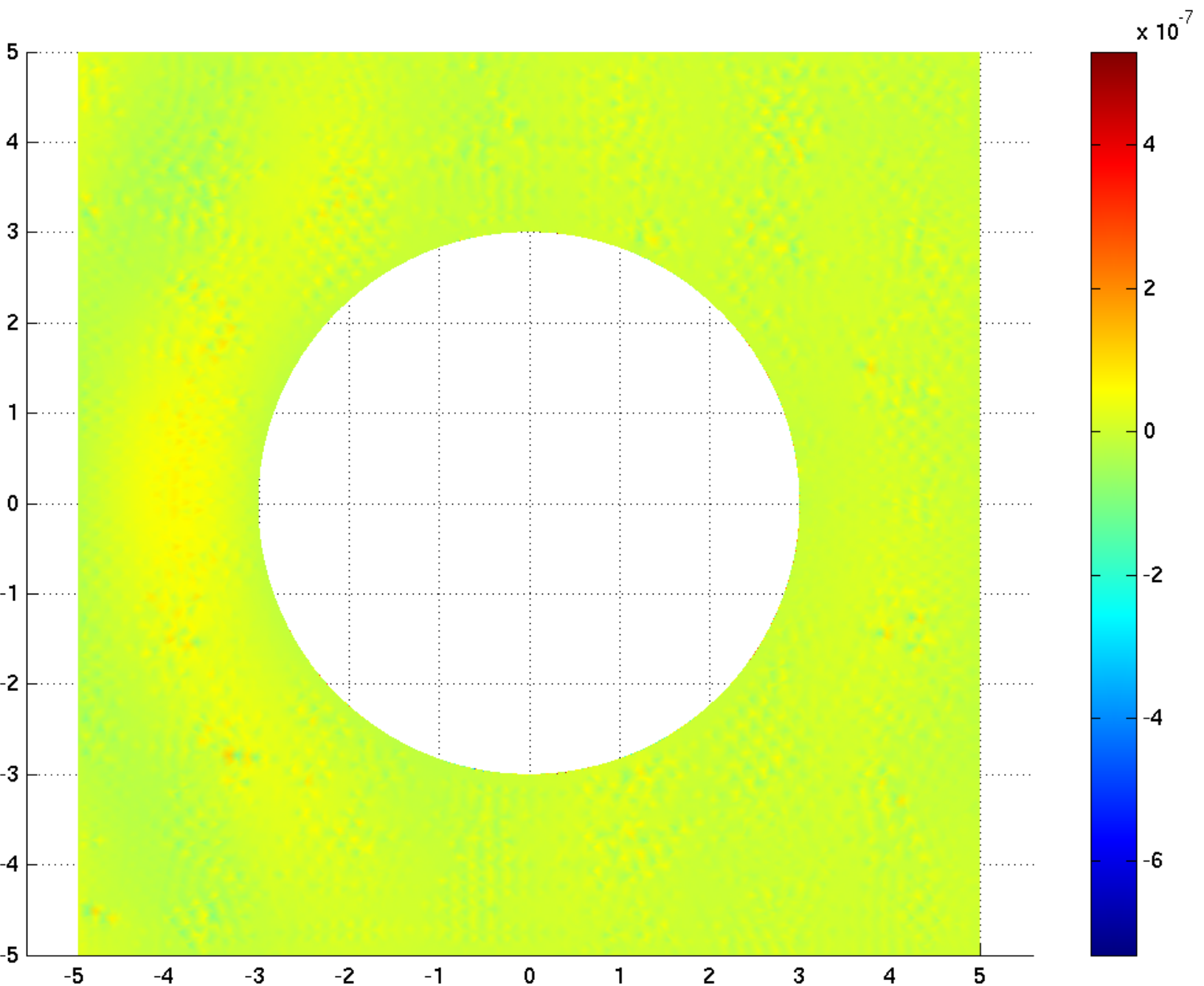}\hfill{}
\includegraphics[width=0.32\textwidth]{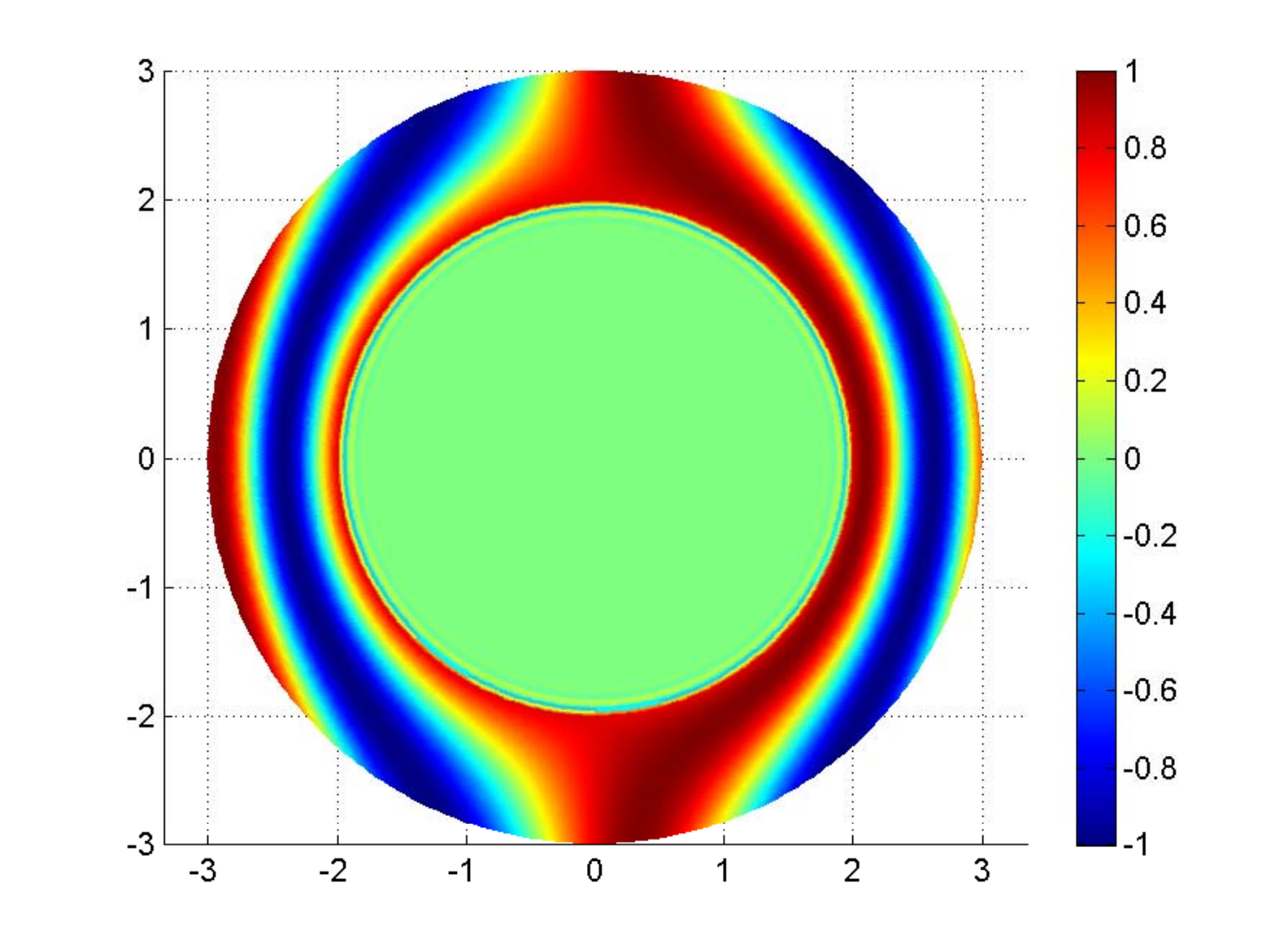}\hfill{}

\caption{\label{fig:S:us:scheme:4} The scattered wave $u^s$ (real
part) and the transmitted wave $u^t$  with respect to $\rho=10^{-1}$
and $10^{-5}$ from top to bottom, respectively, for \emph{Scheme
FSH}.}
\end{figure}


Finally, we study the convergence history of the discrete maximum
norm of $\mathcal{A}(\hat{x})$ with respect to the regularization
parameter $\rho$. From
Figure~\ref{fig:linear:regression:scattering:sh:fsh}, we see clearly
second order decay rate of the discrete maximum norm of
$\mathcal{A}(\hat{x})$ in terms  of $\rho$ for the near-cloak
construction Schemes \emph{SH} and \emph{FSH}, compared with the red
reference line of second order decay in
Figure~\ref{fig:linear:regression:scattering:sh:fsh}, which confirms
the sharpness of our theoretical upper bounds for Schemes \emph{SH}
and \emph{FSH}. Similar to their sound-soft counterparts, it can be
seen from Figure~\ref{fig:linear:regression:scattering:sh:fsh} that
the discrete maximum norms for Schemes \emph{SH} and \emph{FSH} have
nearly the same values as $\rho$ decreases because the \emph{FSH}
lossy layer tends effectively to the sound-hard boundary condition
for small $\rho$ and thus Scheme \emph{FSH} approaches Scheme
\emph{SH} in the limit sense as $\rho\to 0$. The significantly
improved cloaking performance for Schemes \emph{SH} and \emph{FSH}
makes it easier for engineers to design practical near-cloak devices
in a variety of industrial applications.

\begin{figure}
%
%

\hfill{}
\includegraphics[width=0.4\textwidth]{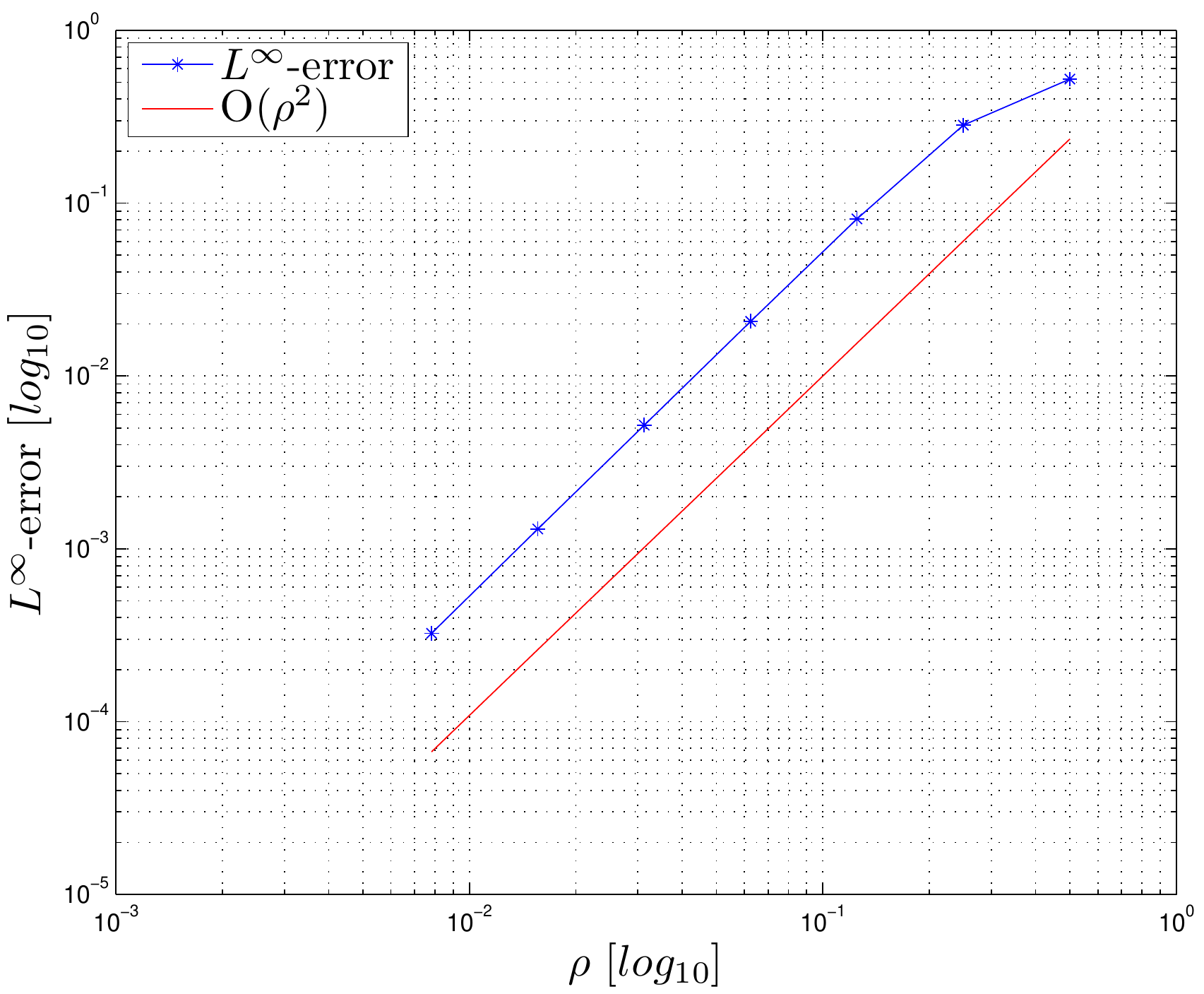}\hfill{}
\includegraphics[width=0.4\textwidth]{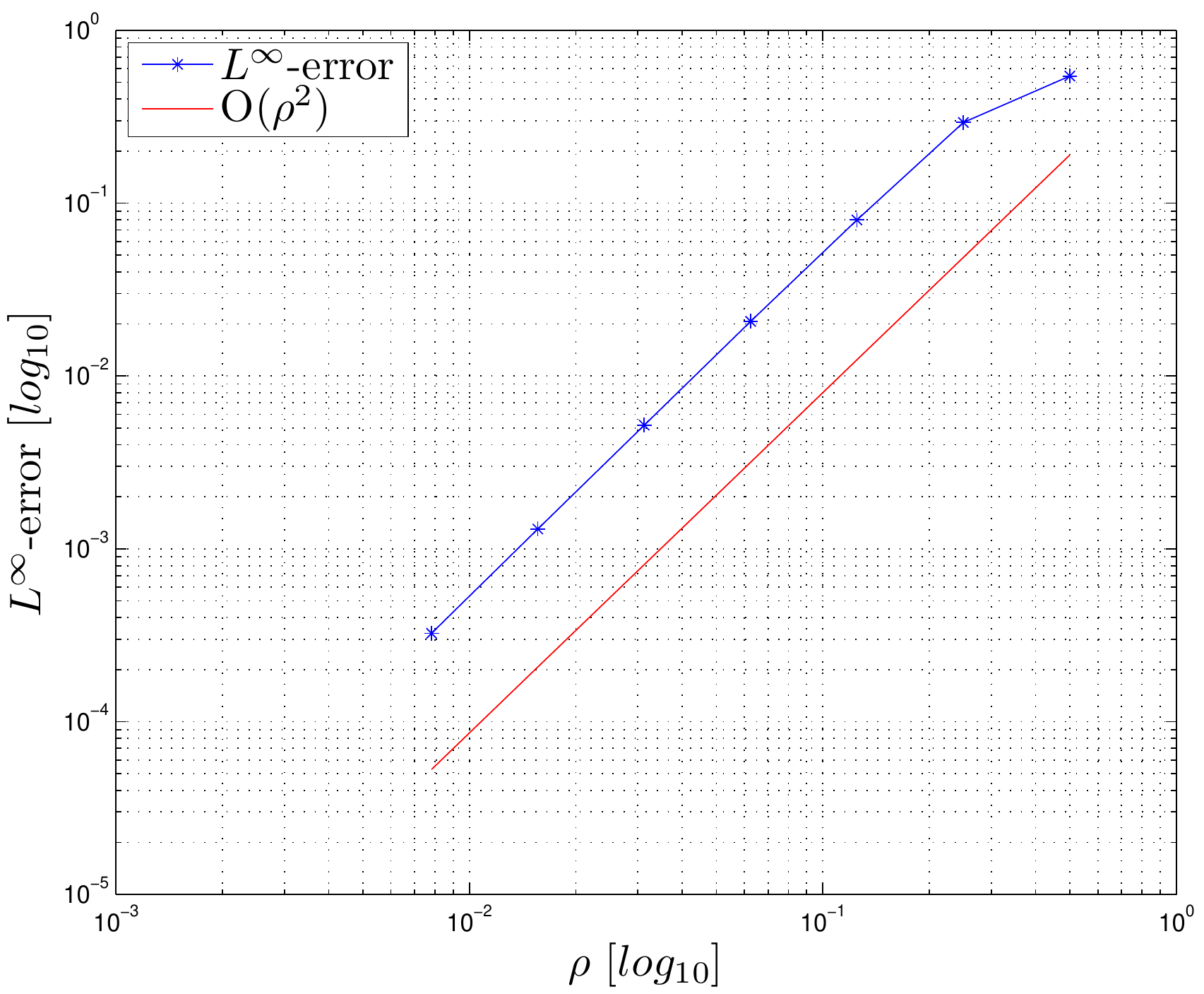}\hfill{}


\caption{\label{fig:linear:regression:scattering:sh:fsh} Convergence
history of scattering measurement data versus $\rho$ for Scheme
\emph{ SH} (Left) and Scheme \emph{ FSH} (Right). }
\end{figure}

\end{document}